\documentclass[aos,preprint,dvipsnames]{imsart}

\usepackage[margin=1.25in]{geometry}
\RequirePackage[OT1]{fontenc}
\RequirePackage{graphicx}

\RequirePackage{amsthm,amsmath,amsfonts,amssymb}

\RequirePackage[cmex10]{amsmath}
\RequirePackage[round, authoryear]{natbib}
\RequirePackage[colorlinks,citecolor=blue,urlcolor=blue]{hyperref}
\usepackage{graphicx}
\usepackage{epstopdf}
\usepackage{amsfonts}
\usepackage{amsmath,amssymb}		
\usepackage{bbm} 	
\usepackage{amsthm}
\usepackage{xcolor}
\usepackage{accents}
\usepackage{easy-todo}
\usepackage{mathtools}
\usepackage{tikz}
\usepackage{paralist} 

\newtheorem{defin}{Definition}[section]
\newtheorem{prop}{Proposition}[section]
\newtheorem{cor}{Corollary}[section]
\newtheorem{theorem}{Theorem}[section]
\newtheorem{lemma}{Lemma}[section]

\newtheorem {Example}{Example}[section]

\numberwithin{equation}{section}

\newcommand{\n}{^{(n)}} 
\newcommand{\indicator}{\mathbbm{1}}     
\newcommand{\mfzero}{\mathbf{0}}         
\newcommand{\pint}{{\boldsymbol{\theta}}}  
\newcommand{\lpint}{{\boldsymbol{\tau}}}   
\newcommand{\pnui}{f}       
\newcommand{\lpnui}{\eta}   

\newcommand{\LPnui}{\Upsilon}
\newcommand{\rN}{\mathbb{N}}             

\newcommand{\lto}{\stackrel{\mathcal{L}}{\longrightarrow}}

\newcommand{\law}{\mathrm{P}}
\newcommand{\lawq}{\mathrm{Q}}

\newcommand{\prob}{\mathbb{P}}
\newcommand{\probn}{\mathbb{P}\n}
\newcommand{\Exp}{\mathbb{E}}
\newcommand{\Expn}{\mathbb{E}\n}
\newcommand{\dd}{\mathrm{d}}             
\newcommand{\outcome}{\Omega}

\newcommand{\E}{\mathcal{E}}                     
\newcommand{\EO}{\E_{\mfzero}}	
\renewcommand{\P}{\mathcal{P}}

\newcommand{\En}{\E\n}
\newcommand{\EnO}{\En_{\mfzero}}

\newcommand{\Field}{\mcB}

\newcommand{\FieldSuf}{\Field_{\text{\rm suff}}}
\newcommand{\FieldSufn}{\Field_{\text{\rm{suff}}}\n}
\newcommand{\FinitePar}{\cal J}
\newcommand{\AField}{\mcA}

\newcommand{\AFieldn}{\AField\n}

\newcommand{\FisherInfo}{{\cal I}}
\newcommand{\FisherInfob}{{\boldsymbol{\cal I}}}

\newcommand{\mcA}{\mathcal{A}}
\newcommand{\mcB}{\mathcal{B}}

\newcommand{\mbR}{\mathbb{R}}

\newcommand{\CS}{\Delta}
\newcommand{\CSb}{\boldsymbol{\Delta}}
\newcommand{\CSshift}{\CSb_{\text{{shift}}}}
\newcommand{\CSdrift}{\CSb_{\text{\rm drift}}}

\newcommand{\s}{u}

\newcommand{\B}{\mathbf{B}}

\newcommand{\Bnui}{\B_{\lpnui}}
\newcommand{\CODF}{{\bf F}_\pm}
\newcommand{\CODFn}{\CODF^{(n)}}

\newcommand{\thetab}{\boldsymbol\theta}
\newcommand{\varthetab}{\boldsymbol\vartheta}

\newcommand{\Thetab}{\boldsymbol\Theta}
\newcommand{\taub}{\boldsymbol\tau}
\newcommand{\Deltab}{\boldsymbol\Delta}
\newcommand{\Gammab}{\boldsymbol\Gamma}

\newcommand{\pr}{^{\prime}}

\usepackage{stackengine}
\stackMath
\newcommand\tenq[2][1]{%
\def\useanchorwidth{T}%
\ifnum#1>1%
\stackunder[0pt]{\tenq[\numexpr#1-1\relax]{#2}}{\scriptscriptstyle\thicksim}%
\else%
\stackunder[1pt]{#2}{\scriptscriptstyle\thicksim}%
\fi%
}

\makeatletter
\newcommand*\rel@kern[1]{\kern#1\dimexpr\macc@kerna}
\newcommand*\widebar[1]{%
  \begingroup
  \def\mathaccent##1##2{%
    \rel@kern{0.8}%
    \overline{\rel@kern{-0.8}\macc@nucleus\rel@kern{0.2}}%
    \rel@kern{-0.2}%
  }%
  \macc@depth\@ne
  \let\math@bgroup\@empty \let\math@egroup\macc@set@skewchar
  \mathsurround\z@ \frozen@everymath{\mathgroup\macc@group\relax}%
  \macc@set@skewchar\relax
  \let\mathaccentV\macc@nested@a
  \macc@nested@a\relax111{#1}%
  \endgroup
}
\makeatother

\graphicspath {{figures/}}

\startlocaldefs
\endlocaldefs

\begin{document}

\begin{frontmatter}
\title{Maximal Ancillarity,  Semiparametric Efficiency, \\ and the 
Elimination of Nuisances}
\runtitle{Ancillarity and Semiparametric Efficiency}


\begin{aug}
\author{\fnms{Marc} \snm{Hallin}\thanksref{m1}
\ead[label=e1]{mhallin@ulb.ac.be}}
\and
\author{\fnms{Bas J.M.} \snm{Werker}\thanksref{m2}
\ead[label=e2]{b.j.m.werker@tilburguniversity.edu}}
\and
\author{\fnms{Bo} \snm{Zhou}\thanksref{m3}
\ead[label=e3]{bzhou@vt.edu}}

\address{\thanksmark{m1}\parbox[t]{0.6\textwidth}{\centering \it 
Université libre de Bruxelles, Belgium  \\  and \\ 
Institute of Information Theory and Automation,
Czech Academy of Sciences, Prague, Czech Republic\smallskip }}
\address{\thanksmark{m2}\parbox[t]{0.4\textwidth}{\centering\it
Tilburg University, The Netherlands}}
\address{\thanksmark{m3}\parbox[t]{0.39\textwidth}{\centering\it
Virginia Tech, Blacksburg, VA, USA}}

\runauthor{Hallin, Werker, and Zhou}
\end{aug}

\printaddresses

\begin{abstract} 
Restricting statistical experiments via nuisance-ancillary $\sigma$-fields yields  nuisance-free experiments. However, a moot point with ancillarity is that maximal ancillary $\sigma$-fields are typically not unique. There are exceptions, though, among which the limiting experiments in a locally asymptotically normal (LAN) context. Building on this, we address the maximal ancillarity uniqueness problem by adopting a H\' ajek-Le Cam asymptotic perspective and define the concept of sequences of locally asymptotically maximal nuisance-ancillary\linebreak $\sigma$-fields. We then show that any semiparametrically efficient procedure admits versions that are measurable with respect to  such\linebreak  $\sigma$-fields while enjoying strict  finite-sample nuisance-ancillarity, hence eliminating the nuisance without the hassle of estimating it. This is in sharp contrast with classical tangent space projections, which also achieve semiparametric efficiency but only enjoy asymptotic nuisance-ancillarity---at the price, moreover, of  adequately estimating  the nuisance. When the nuisance is the density of some noise or innovation driving the data-generating process of a LAN experiment, we show that a sequence of locally asymptotically maximal nuisance-ancillary $\sigma$-fields is generated by the so-called center-outward residual ranks and signs based on measure transportation results. Restricting local experiments to such $\sigma$-fields yields   sequences of finite-sample nuisance-free (here, distribution-free)  restrictions of the original local LAN experiments that nevertheless achieve the semiparametric efficiency bounds of the ori\-ginal ones.
\end{abstract}



\end{frontmatter}

\section{Introduction}

\noindent Besides a parameter of interest $\thetab$, most statistical experiments of practical interest also involve a nuisance parameter $\varthetab$ which, moreover, often is infinite-dimensional. Typical examples are the semiparametric experiments where the nuisance is the unspecified density $f$ of some noise or innovation driving the data-generating process.

\subsection{Ancillarity and the elimination of nuisance parameters}
Eliminating nuisances---identifying nuisance-free subexperiments to perform nuisance-free inference---has been a central problem in statistics for almost a century. A basic concept, in that context, is  {\it ancil\-larity}, which can be traced back to \cite{Fisher25}, where the term was coined---see \cite{Stig01, Stig92}. Formal definitions and profound properties were provided much later, with fundamental contributions by Debabrata Basu \citep{basu1955, basu1958, basu1959, basu1964}; see \cite{Dawid2011}. Various refinements of the concept, showing the complexity of the notion, also can be found in \cite{Fraser56} and  \cite{BN73, BN76, BN99},  among others;  see \cite{Soren24} for details.
  
Basu's 1977 classical survey of the problem \citep{basu1977} mainly discusses the fundamental notions of {\it completeness}, {\it (minimal) sufficiency}, and {\it (maximal) ancillarity} and the related methods (marginalization, conditioning,...) of nuisance elimination, but comes to the somewhat distressing conclusion that no general solution, applicable to arbitrary models, exists. More recent surveys such as \cite{Ghoshetal.10} more or less end up with the same conclusion. Although the problem of eliminating nuisances has been an active line of research for many years, ancilla\-rity remains {\it ``a shadowy topic in statistical theory''} \citep{Cox82} that {\it ``continues to intrigue''} \citep{Ghoshetal.10} and, despite {\it ``an astoni\-shingly large amount of effort and ingenuity''}, still is {\it ``strewn with logicians' nightmares''} \citep{basu1977}. 
  
In this paper, we are revisisting the problem of ancillarity and nuisance elimination from a local and asymptotic point of view in the context of locally asymptotically normal (LAN) experiments and closely related issues of semiparametric efficiency. 

\subsection{Maximal nuisance-ancillarity}
In a statistical experiment
\begin{equation}\label{exp}
{\mathcal E}\coloneqq \big({\mathcal X}, {\mathcal B}, {\mathcal P}\coloneqq\{{\rm P}_{\thetab , \varthetab} :  (\thetab , \varthetab ) \in {\boldsymbol\Theta}\times{\mathcal F}\}
\big),
\end{equation}
where~$\thetab$ is a finite-dimensional parameter of interest and~$\varthetab$ a (possibly infinite-dimensional) nuisance parameter,
call $\cal E$-{\it nuisance-ancillary at}~$\thetab$ any measurable function $({\bf X}, \thetab)\mapsto {\boldsymbol\zeta}^\dag_{\thetab}({\bf X})$ the distribution of which, under ${\rm P}_{\thetab,\varthetab}$, does not depend on~$\varthetab$. That ancillarity property, clearly, is invariant under (measurable) transformations: rather than the property of a single variable, nuisance-ancillarity is the property of a $\sigma$-field, and we call $\cal E$-{\it nuisance-ancillary at}~$\thetab$ any $\sigma$-field  ${\mathcal B}^\dag_{{\thetab}}\subseteq{\cal B}$ of events the probability of which, under~${\rm P}_{\thetab,\varthetab}$, does not depend on~$\varthetab$ (in Basu's terminology, a {\it $\thetab$-oriented $\sigma$-field}). In other words, ${\mathcal B}^\dag_{{\thetab}}$ is~$\cal E$-nuisance-ancillary at~$\thetab$ if and only if it is ancillary in the specified-$\thetab$ subexperiment
\begin{equation}
{\mathcal E}(\thetab)\coloneqq \big({\mathcal X}, {\mathcal B}, {\mathcal P}_{\thetab}\coloneqq\{{\rm P}_{\thetab , \varthetab} :  \varthetab  \in {\mathcal F}\} \big),\quad  \thetab\in {\boldsymbol\Theta}
\end{equation}
of~${\mathcal E}$---a {\it local} concept, thus, localized at $\thetab$. 

When such ${\mathcal B}^\dag_{{\thetab}}$ exist, a natural way of eliminating the nuisance locally (at $\thetab_0$) is to base inference about~$\thetab$ on~${\mathcal B}^\dag_{{\thetab_0}}$-measurable procedures: nuisance-free tests of hypotheses of the form~${\cal H}_0:\thetab = \thetab_0$, for instance, should be based on ${\mathcal B}^\dag_{{\thetab_0}}$-measurable test statistics; confidence regions for $\thetab$ at confidence level $(1-\alpha)$ should be constructed as collections of $\thetab_0$ values such that some~${\mathcal B}^\dag_{{\thetab_0}}$-measurable test does not reject ${\cal H}_0:\thetab = \thetab_0$  at level $\alpha$, etc. 

The larger ${\mathcal B}^\dag_{{\thetab}}$, the better\footnote{``Better'' here is meant in the sense of (semiparametric) efficiency; see Section~\ref{semiparameff} for a rigorous treatment.}: only {\it maximal} nuisance-ancillary $\sigma$-fields ${\mathcal B}^\dag_{{\thetab}}$ (maximal here is to be understood in the sense of set inclusion up to $\{{\rm P}_{\thetab , \varthetab} : (\thetab, \varthetab)\in\Thetab\times{\cal F}\}$-null sets), thus, should be considered. Unfortunately, it is a well-known fact that maximal ancillary $\sigma$-fields typically are not unique (see, among many others, \cite{basu1959, basu1964, Cox71}), and choosing one of them is everything but obvious: which ones are doing the ``best job''? Which ones are ``optimally'' preserving all the available information about $\thetab$? This issue with the concept is well-documented, if not resolved, in classical textbooks (see, e.g., Chapter~10 in \cite{LehmannRomano2005}).

Exceptions exist, though: a unique\footnote{When such a unique maximal nuisance-ancillary $\sigma$-field exists, we will call it \emph{strongly maximally nuisance-ancillary}.} maximal nuisance-ancillary $\sigma$-field exists, for instance, in the $d$-dimensional Gaussian location model (one $d$-dimensional observation $\bf X$) with specified full-rank covariance---the so-called {\it Gaussian shift} experiment ${\cal E}_{\text{\rm shift}}$---where (with~$d=d_1 + d_2$)
$${\bf X}\coloneqq \left(\begin{array}{c}{\bf X}_1\\ {\bf X}_2\end{array}\right)\sim{\mathcal N}\left( \left(\begin{array}{cc} {\boldsymbol \Sigma}_{11} & {\boldsymbol \Sigma}_{12} \\ {\boldsymbol \Sigma}_{12}^\top  &{\boldsymbol \Sigma}_{22}
\end{array}\right)\!\left(\begin{array}{c}\thetab \\ \varthetab\end{array}\right)
 ,\, \left(\begin{array}{cc} {\boldsymbol \Sigma}_{11} & {\boldsymbol \Sigma}_{12} \\ {\boldsymbol \Sigma}_{12}^\top  &{\boldsymbol \Sigma}_{22}
\end{array}\right)\right) , \quad \left(\begin{array}{c}\thetab \\ \varthetab\end{array}\right)\in {\mathbb R}^{d_1+d_2}.
$$
Here, the $\sigma$-field ${\cal B}^*\subseteq{\cal B}^d$ (${\cal B}^d$ the Borel $\sigma$-field on ${\mathbb R}^d$) generated by
$${\bf X}_1^*\coloneqq {\bf X}_1 - {\boldsymbol \Sigma}_{12}{\boldsymbol \Sigma}_{22}^{-1}{\bf X}_2\sim{\cal N}\Big(\big({\boldsymbol \Sigma}_{11}- {\boldsymbol \Sigma}_{12}{\boldsymbol \Sigma}_{22}^{-1}{\boldsymbol \Sigma}_{12}^\top\big)\thetab ,\, {\boldsymbol \Sigma}_{11} - {\boldsymbol \Sigma}_{12}{\boldsymbol \Sigma}_{22}^{-1}{\boldsymbol \Sigma}_{12}^\top\Big)$$ 
is the unique\footnote{To avoid repeating {\it essentially unique} (that is, {\it unique up to sets of ${\rm P}_{\thetab ,\varthetab}$-probability zero for all~$(\thetab,\varthetab)$}) again and again, we tacitly include such null sets in all maximal ancillary $\sigma$-fields} maximal nuisance-ancillary $\sigma$-field at (each) $\thetab$ (see Section~\ref{Sec23}).


The importance of Gaussian shift experiments in this context stems from their role as local limiting experiments in the class of Locally Asymptotically Normal (LAN) experiments.
However, the uniqueness, in the limiting Gaussian shift experiment, of a maximal nuisance-ancillary~$\sigma$-field, does not imply the uniqueness of maximal nuisance-ancillary $\sigma$-fields in the corresponding sequences of local experiments, where, for finite $n$, several maximal nuisance-ancillary $\sigma$-fields may coexist. As a result, it is not clear how optimal inference based on ancillarity arguments in the limiting Gaussian shift experiment can be transported to the sequence of (local) experiments. Before we explain how we resolve this issue, let us give an example.

\begin{Example}\label{ex1}{\rm 
A typical example---studied in detail in Section~\ref{Sec4}---is the case of models with unspecified {\it residual} or {\it innovation} density---namely, experiments of the form 
\begin{equation}\label{unspecifiedf}
{\mathcal E}\coloneqq\Big({\mathbb R}^{n\times d},
{\mathcal B}^{n\times d},  {\cal P}\n\coloneqq \{{\rm P}_{\thetab , f}^{ (n)} : \thetab\in\Thetab , f\in{\cal F}\}
\Big),
\end{equation}
where ${\mathcal B}^{n\times d}$ stands for the Borel $\sigma$-field on ${\mathbb R}^{n\times d}$, $\Thetab\subseteq{\mathbb R}^k$ for some open $k$-dimensional real parameter space,\footnote{For simplicity of notation, we throughout assume $\Thetab = {\mathbb R}^k$.} and $\cal F$ for some broad class of densities to be specified later on---or, since we are to adopt an asymptotic point of view, sequences of experiments
\begin{equation}\label{unspecifiedfn}
{\mathcal E}\n_{\text{\rm global}}\coloneqq\Big({\mathbb R}^{n\times d},
{\mathcal B}^{n\times d}, {\cal P}\n\coloneqq\{{\rm P}_{\thetab , f}^{( n)} : \thetab\in\Thetab , f\in{\cal F}\}
\Big), \quad n\in\mathbb{N},
\end{equation}
where $\thetab$ is a parameter of interest and $f\in{\mathcal F}$, the unspecified density of some {\it residual} or {\it innovation} driving the data-generating process, plays the role of the nuisance. More precisely, the observation, in such a model, is an $n$-tuple~${\bf X}\n = ({\bf X}\n_1,\ldots, {\bf X}\n_n)$ where~${\bf X}\n_i$ is~${\mathbb R^d}$-valued, and there exists a mapping 
\begin{equation}\label{resid}
({\bf X}\n , \thetab ) \mapsto {\bf Z}\n(\thetab )\coloneqq ({\bf Z}\n_1(\thetab ),\ldots, {\bf Z}\n_n(\thetab)),
\end{equation}
from~${\mathbb R}^{n\times d}\times\Thetab$ to ${\mathbb R}^{n\times d}$ (the {\it residual function}\footnote{In time series models, the assumption that the range of the residual function is ${\mathbb R}^{n\times d}$ has to be slightly relaxed due to the dependence of residuals on initial values: for a stationary VAR model of order $p$, for instance, the residual function takes values in ${\mathbb R}^{(n-p)\times d}$ and the correspondence between ${\bf X}\n$ and the residuals involves a $p$-tuple of initial values (the asymptotic impact of the latter, however, is nil).}) such that ${\bf X}\n \sim~\!{\rm P}\n_{\thetab,f}$ if and only if the residuals~${\bf Z}\n_1(\thetab ),\ldots, {\bf Z}\n_n(\thetab )$ are i.i.d.\ with density $f$. This class of experiments covers a large variety of classical semiparametric models such as   single- and multiple-output linear and nonlinear regression,
panel data, MANOVA, VARMA, long-memory, nonlinear, and cointegrated time-series, etc.\ with unspecified error or innovation density. Call such   ${\cal E}\n_{\text{\rm global}}$ a (sequence of) {\it unspecified density experiment(s)}.

Nuisance-ancillarity, in this context, is thus {\it distribution-freeness} with respect to $f\in{\mathcal F}$. It can be shown\footnote{This follows from a theorem by \cite{basu1959}; see Appendix E, Corollary E.1 in the online supplement to Hallin et al. (2021).} that, for $d>1$, the $\sigma$-field generated by the ranks of the first components of the residuals~${\bf Z}\n_i(\thetab )$ is maximal ${\mathcal E}\n$-nuisance-ancillary  at $\thetab$. But so is the $\sigma$-field generated by the ranks of~${\bf Z}\n_i(\thetab )$'s second components, by the ranks of ${\bf Z}\n_i(\thetab )$'s third components, ..., by the ranks of ${\bf Z}\n_i(\thetab )$'s $d$-th components. This yields $d$ distinct maximal nuisance-ancillary $\sigma$-fields, none of which can be enlarged without losing its ancillarity properties (in particular, they are not jointly ancillary) while all of them are overlooking a significant amount of ``distribution-free information'' about the parameter of interest $\thetab$. Which one, if any, is ``best''? As we shall see, asymptotics may offer a solution thanks to the fact that, under Local Asymptotic Normality (LAN), limiting local experiments, \emph{when chosen wisely}, unlike the finite-$n$ ones, admit a \emph{unique} maximal ancillary $\sigma$-field.}
\end{Example}

Our approach to the non-uniqueness of maximal nuisance-ancillary $\sigma$-fields in the sequence of local experiments, is to reformulate the limiting experiment. More precisely,  we extend the limiting {\it Gaussian shift} experiments to  {\it Brownian drift} ones. These Gaussian shift and Brownian drift experiments are equivalent from the point of view of the Le Cam distance: hence, they both qualify, under LAN, as local limiting experiments. Brownian drift experiments, however, are defined on richer $\sigma$-fields. This allows us to analyze the non-uniqueness of maximal nuisance-ancillary $\sigma$-fields in the sequence of local experiments, and choose the ``best'', i.e., the one that converges (in a sense to be made precise) to the unique maximal nuisance-ancillary $\sigma$-field in the (Brownian drift) limit experiment. In the previous example, this will naturally lead us to choose the $\sigma$-field generated by the so-called center-outward residual ranks and signs, see Section~\ref{Sec4}.

\subsection{Outline of the paper} The main contribution of this paper is as follows. Denote by
\begin{equation}\label{localE}
{\mathcal E}\n_{\thetab_0 , \varthetab_0}\coloneqq\Big({\mathcal X}\n , {\mathcal B}\n, {\mathcal P}\n_{\thetab_0 , \varthetab_0} \coloneqq\{{\mathbb P}\n_{\taub , \eta} :  (\taub , \eta)\in 
{\mathbb R}^k\times\Upsilon_{\varthetab_0}\}\Big),\quad n\in {\mathbb N},
\end{equation}
the sequence  of {\it local experiments} localized at $(\thetab_0, \varthetab_0)$, with {\it local parameter} $(\taub ,\eta)$ ($\taub$ the local parameter of interest, $\eta$ the local nuisance) of a LAN sequence
\begin{equation}\label{globalE}
{\mathcal E}\n_{\text{\rm global}}\coloneqq\Big({\mathcal X}\n , {\mathcal B}\n, {\mathcal P}\n =\{{\rm P}\n_{\thetab , \varthetab} :   (\thetab , \varthetab ) \in 
{\mathbb R}^k\times{\mathcal F} \Big) ,\quad n\in {\mathbb N}
\end{equation}
of experiments with parameters $(\thetab, \varthetab)$.
As $n\to\infty$, that local sequence ${\mathcal E}\n_{\thetab_0,\varthetab_0}$ converges, in the Le Cam distance, to a limit experiment usually described as a Gaussian shift ${\cal E}_{{\scriptstyle{\text{\rm shift;}} {(\thetab_0, \varthetab_0)}}}$. This limiting Gaussian shift, however, can be replaced by an equivalent (in the sense of the Le Cam distance) limiting Brownian drift experiment ${\cal E}_{{\scriptstyle{\text{\rm drift;}} {(\thetab_0, \varthetab_0)}}}$ that  
admits a uni\-que~${\cal E}_{{\scriptstyle{\text{\rm drift;}} {(\thetab_0, \varthetab_0)}}}$-ancillary at $\taub = {\boldsymbol 0}$ $\sigma$-field ${\cal B}_{\scriptstyle{\text{\rm drift;}\scriptstyle{\boldsymbol 0}}}^\ddagger$.

However, unlike the Gaussian shift limit, we can now introduce, see Definition~\ref{weakcv}, for each $\thetab_0\in{\mathbb R}^k$, sequences~${\mathcal B}^{\ddagger (n)}_{\thetab_0}$,~$n\in{\mathbb N}\vspace{.5mm}$ of \emph{strongly maximal ${\cal E}\n_{\text{\rm global}}$-nuisance-ancillary at}~$\thetab_0$ sub-$\sigma$-fields of ${\mathcal B}\n$ with the property that (i)~each ${\mathcal B}^{\ddagger (n)}_{\thetab_0}$ is, for fixed~$n$,~${\cal E}\n_{\text{\rm global}}$-maximal nuisance-ancillary at $\thetab_0$ and (ii)~allows, as~$n\to\infty$, for a ``reconstruction'' of ${\cal B}_{\scriptstyle{\text{\rm drift;}\scriptstyle{\boldsymbol 0}}}^\ddagger$, the \emph{unique} maximal nuisance-ancillary $\sigma$-field in ${\cal E}_{{\scriptstyle{\text{\rm drift;}} {(\thetab_0, \varthetab_0)}}}$.

We then show that, when such sequences exist, semiparametric efficiency, at any $(\thetab,\varthetab)$ in~${\mathbb R}^k\times {\cal F}$, can be achieved via~${\mathcal B}^{\ddagger (n)}_{\thetab}$-measurable---hence, finite-$n$ nuisance-free---procedures. This is in sharp contrast with traditional semiparametric methods, based on tangent space projections, where nuisance-freeness is only asymptotic. 

To conclude, we show (Section~\ref{Sec4}) that, in the unspecified density model \eqref{unspecifiedfn},  the sequence of $\sigma$-fields generated by the measure-transportation-based {\it center-outward ranks and signs} of the $\thetab_0$-residuals~\eqref{resid}  is   strongly maximal ${\cal E}\n_{\text{\rm global}}$-nuisance-ancillary at    $\thetab_0$. This implies that the semiparametrically efficiency bounds at~$(\thetab_0, \varthetab_0)$ can be reached via fully distribution-free procedures based on center-outward ranks and signs.

\section{Locally asymptotically normal (LAN) sequences of experiments}
\subsection{Global and local experiments} 
Throughout, we consider the sequence ${\cal E}\n_{\text{\rm global}}$ of global experiments~\eqref{globalE} and the sequences ${\cal E}\n_{\thetab_0, \varthetab_0}$ of local experiments  defined in~\eqref{localE} where, for simplicity, we assume that $\thetab$, $\thetab_0$, and $\taub$ range over $ {\mathbb R}^k$, $\varthetab$ and $\varthetab_0$ over some separable infinite-dimensional vector space $\cal F$, and $\eta$ over some separable  infinite-dimensional vector space~$\Upsilon_{\varthetab_0}$ that possibly depends on $\varthetab_0$. The local (at $(\thetab_0 , \varthetab_0)$) and global proba\-bility distributions are related 
by~${\mathbb P}\n_{\taub,\eta}\coloneqq {\rm P}\n_{{\mathfrak t}\n(\taub), {\mathfrak f}\n(\eta)}$ for some sequences of bijective mappings~${\mathfrak t}\n={\mathfrak t}\n_{\thetab_0}$ from~$ {\mathbb R}^k$ to ${\mathbb R}^k$ and ${\mathfrak f}\n={\mathfrak f}_{\varthetab_0}\n$ from $\Upsilon_{\vartheta_0}$ to $\cal F$ such that ${\mathfrak t}_{\thetab_0}\n({\boldsymbol 0}) = \thetab_0$ and, denoting by~${\boldsymbol 0}_{\Upsilon_{\varthetab_0}}$ the origin in~$\Upsilon_{\varthetab_0}$, ${\mathfrak f}\n _{\varthetab_0}({\boldsymbol 0}_{\Upsilon_{\varthetab_0}}) = \varthetab_0$ for all~$n$. In traditional situations, we typically have~${\mathfrak t}_{\thetab_0}\n(\taub)$ of the form~$\thetab_0 + n^{-1/2}\taub$ and, for some smooth mapping ${\mathfrak f}_{\varthetab_0}$ from $\Upsilon_{\varthetab_0}$ to $\cal F$ satisfying~${\mathfrak f}({\boldsymbol 0}_{\Upsilon_{\varthetab_0}}) = \varthetab_0$, ${\mathfrak f}\n _{\varthetab_0}(\eta)= {\mathfrak f}_{\varthetab_0}(n^{-1/2}\eta)$ so that, in~\eqref{localE}, ${\mathbb P}\n_{\taub,\eta} ={\rm P}\n_{\thetab_0 +n^{-1/2}\taub, {\mathfrak f}_{\varthetab_0}(n^{-1/2}\eta) }$. 

Note that, for fixed $n$,  the experiments ${\cal E}\n_{\text{\rm global}}$ and ${\mathcal E}\n_{\thetab_0 , \varthetab_0}$ are sharing the same observation spaces and the same families of distributions: ${\cal P}\n$ and ${\cal P}\n_{\thetab_0 , \varthetab_0}$, indeed, only differ by their parametrizations:  $(\thetab , \varthetab)$ for ${\cal P}\n$, $(\taub, \eta )$ for ${\cal P}\n_{\thetab_0, \varthetab_0}$, where the correspondence between~$(\thetab , \varthetab)$ and~$(\taub, \eta )$ is bijective (but depends on $\thetab_0, \varthetab_0$, and $n$). In particular, 
$${\mathfrak f}\n_{\varthetab_0}(\Upsilon_{\varthetab_0}) = {\cal F} \quad\text{and}\quad  {\cal P}\n_{\thetab_0}\coloneqq \{{\rm P}\n_{\thetab_0 , \varthetab} : \varthetab\in{\cal F}\} = \{{\mathbb P}\n_{{\boldsymbol 0}, \eta} : \eta\in\Upsilon_{\varthetab_0}\}, \quad \text{for all $n\in {\mathbb N}$}.
$$
In view of the definition of nuisance-ancillarity, we have the following lemma.

\begin{lemma} A sub-$\sigma$-field ${\cal B}\n_{\thetab}$ of  ${\cal B}\n$ is  ${\cal E}\n_{\thetab , \varthetab}$-nuisance-ancillary at $\taub = {\boldsymbol 0}$ if and only if it is~${\cal E}\n_{\text{\rm global}}$-nuisance-ancillary at $\thetab$.
\end{lemma}

\subsection{Limiting experiments: from Gaussian shift to Brownian drift}\label{Sec22}
Under LAN, the sequences ${\cal E}\n_{\thetab_0, \varthetab_0}$ of local experiments  defined in~\eqref{localE}  converge {\it weakly}, in the Le Cam distance, to limiting experiments of the form 
\begin{equation}\label{limitexp}
{\cal E}_{\thetab_0, \varthetab_0}\coloneqq \Big\{{\cal X}, {\cal B}, {\cal P}\coloneqq\{{\mathbb P}_{\taub,\eta}: \taub\in{\mathbb R}^k, \eta\in \Upsilon_{\varthetab_0}
\}
\Big\}.\end{equation}
The limiting experiment ${\cal E}_{\thetab_0, \varthetab_0}$, however, is not uniquely defined (the Le Cam distance is only a {\it pseudo-distance}). For any two probability measures $\rm P$ and $\rm Q$ defined over some common probability space, denote,  as usual, by $\rm dP/dQ$ the Radon-Nikodym derivative of the component of $\rm P$ which is absolutely continuous with respect to $\rm Q$: then, any experiment 
\begin{equation}\label{limitexp'}
{\cal E}\pr_{\thetab_0, \varthetab_0}\coloneqq \Big\{{\cal X}\pr, {\cal B}\pr, {\cal P}\pr\coloneqq\{{\mathbb P}\pr_{\taub,\eta}: \taub\in{\mathbb R}^k, \eta\in \Upsilon_{\varthetab_0}
\}
\Big\}\end{equation}
such that the joint distribution, under ${\mathbb P}\pr_{{\boldsymbol 0}, 0}$, of the $\ell$-tuple of log-likelihood ratios
\[
\big\{\log({\rm dP}\pr_{\taub_1, \eta_1}/{\rm dP}\pr_{{\boldsymbol 0}, 0}),\ldots, \log({\rm dP}\pr_{\taub_\ell, \eta_\ell}/{\rm dP}\pr_{{\boldsymbol 0}, 0})
\big\}
\]
coincides, for all (finite) $\ell $ and all  $(\taub_1,\eta_1),\ldots,(\taub_\ell,\eta_\ell)$ in ${\mathbb R}^k\times \Upsilon_{\varthetab_0}$, with the joint distribution, under ${\mathbb P}_{{\boldsymbol 0}, 0}$, of the  $\ell$-tuple 
\[
\big\{\log({\rm dP}_{\taub_1, \eta_1}/{\rm dP}_{{\boldsymbol 0}, 0}),\ldots, \log({\rm dP}_{\taub_\ell, \eta_\ell}/{\rm dP}_{{\boldsymbol 0}, 0})
\big\}
\]
is equivalent, in the sense of the Le Cam distance, to  ${\cal E}_{\thetab_0, \varthetab_0}$. Hence, ${\cal E}\pr_{\thetab_0, \varthetab_0}$ also constitutes a limit experiment for the local sequence ${\cal E}\n_{\thetab_0, \varthetab_0}$.

Under LAN, the limiting experiment ${\cal E}_{\thetab_0, \varthetab_0}$ is usually described as a Gaussian shift experiment ${\cal E}_{\text{shift}} = {\cal E}_{{\scriptstyle{\text{\rm shift;}} {(\thetab_0,\varthetab_0)}}}$ with observation 
\begin{equation*}
{\Deltab}_{\text{{\rm{shift}}}} = {\Deltab}_{{
 \scriptstyle{\text{\rm shift;}} {(\thetab_0, \varthetab_0)}}}
\coloneqq 
\left(\CSb_{\text{\rm int}} ^\top\coloneqq\left(\CS_{
{\text{\rm int}},1} ,\ldots, \CS_{{\text{\rm int}},k}\right),
\left\{ \CS_{{\text{nuis}},\lpnui}\big\vert\lpnui\in\LPnui\right\}\right)^\top; 
\end{equation*}
since $\Upsilon_{\varthetab_0}$ here is infinite-dimensional, ${\Deltab}_{\text{{\rm{shift}}}}$ is a Gaussian process rather than a Gaussian vector,  with, under~$\prob_{\lpint ,\lpnui}$, for any~$\ell\in\mathbb{N}$ and any~$(\lpint,  \lpnui_1,\ldots,\lpnui_\ell)\in{\mbR^k}\times\LPnui^\ell_{\varthetab_0}$, full-rank (Gaussian) finite-dimensional marginals\vspace{1mm} 
\begin{equation}\label{margDshift}
\Deltab_{\text{shift}}^{(\ell)}	\coloneqq\left(\begin{array}{c}\CSb_{\text{\rm int}} \eqqcolon \Deltab_I \\
		\CS_{{\text{nuis}},\eta_1}\eqqcolon\CS_{I\!I,\lpnui_1}\\  \vdots \\  \CS_{{\text{nuis}},\eta_\ell}\eqqcolon \CS_{I\!I,\lpnui_{\ell}}
	\end{array}
	\right)  \sim{\cal N}
	\left(\left(\begin{array}{c}\FisherInfob_{I\!,I}\lpint +\FisherInfob_{I\!, \lpnui}
		\\ 
		\FisherInfob_{I\!, \lpnui_1}^\top\lpint+\FisherInfo_{\lpnui_1 \lpnui} 
		\\ 
		\vdots 
		\\  
		\FisherInfob_{I\!, \lpnui_{\ell}}^\top\lpint+\FisherInfo_{\lpnui_{\ell} \lpnui} 
	\end{array} \right),\,\,  \FisherInfob_{\eta_1,\ldots,\eta_{\ell}}\coloneqq 
	\left(\begin{array}{cccc}
		\FisherInfob_{I\!,I}&\FisherInfob_{I\!, \lpnui_1}&\!\!\!\!\ldots\!\!\!\!&\FisherInfob_{I\!, \lpnui_{\ell}}\\ 
		\FisherInfob_{I\!, \lpnui_1}^\top&\FisherInfo_{\lpnui_1 \lpnui_1}&\!\!\!\!\ldots\!\!\!\!&\FisherInfo_{\lpnui_1 \lpnui_{\ell}}
		\\
		\vdots&&\!\!\!\!\ldots\!\!\!\!&
		\\
		\FisherInfob_{I\!, \lpnui_{\ell}}^\top&\FisherInfo_{\lpnui_{\ell} \lpnui_1}&\!\!\!\!\ldots\!\!\!\!&\FisherInfo_{\lpnui_{\ell} \lpnui_{\ell}}
	\end{array}\right)
	\right).
\end{equation}
The $k\times k$ information matrix 
$\FisherInfob_{I\!,I}=\FisherInfob_{I\!,I}(\pint_0, \varthetab_0)$, the $k\times 1$ cross-information quan\-tities~$\FisherInfob_{I\!, \eta} = \FisherInfob_{I\!, \lpnui} (\pint_0, \varthetab_0)$, $\eta\in\Upsilon_{\varthetab_0}$, and  the scalar information and cross-information quanti\-ties~$\FisherInfo_{\lpnui \lpnui^\prime}=\FisherInfo_{\lpnui \lpnui^\prime} (\pint_0, \varthetab_0)$,  $(\lpnui,\,  \lpnui^\prime)\in\Upsilon^2_{{\varthetab_0}}$ depend on  the particular LAN experiment under consideration; for~$\lpnui = 0$, $\FisherInfob_{I\!, 0}=\mfzero$ and~$\FisherInfo_{ 0\, \lpnui^\prime}=0$ for all $\lpnui^\prime\in\Upsilon_{\varthetab_0}$.  
Under $\prob_{\mfzero,0}$, $\Deltab_{\text{\rm shift}} $ is full-rank Gaussian and has finite-dimensional Gaussian marginals with mean $\boldsymbol 0$ and  (since~$ \FisherInfob_{\eta_1,\ldots,\eta_{\ell}}$ does not depend on $(\taub , \eta)$)  the same covariance matrix $ \FisherInfob_{\eta_1,\ldots,\eta_{\ell}}$ as in \eqref{margDshift}\vspace{1mm}:
\begin{equation}\label{Normal0}
	\Deltab_{\text{shift}}^{(\ell)} = \left(\begin{array}{c}\CSb_{\text{\rm int}} \eqqcolon \Deltab_I \\
		\CS_{{\text{nuis}},\eta_1}\eqqcolon\CS_{I\!I,\lpnui_1}\\  \vdots \\  \CS_{{\text{nuis}},\eta_\ell}\eqqcolon \CS_{I\!I,\lpnui_{\ell}}
	\end{array}
	\right)\sim {\cal N}\left(\mfzero,\,  \FisherInfob_{\eta_1,\ldots,\eta_{\ell}}= \left(\begin{array}{cccc}
		\FisherInfob_{I\!,I}&\FisherInfob_{I\!, \lpnui_1}&\ldots&\FisherInfob_{I\!, \lpnui_k}\\ 
		\FisherInfob_{I\!, \lpnui_1}^\top&\FisherInfo_{\lpnui_1 \lpnui_1}&\ldots&\FisherInfo_{\lpnui_1 \lpnui_k}
		\\
		\vdots&&\ldots&
		\\
		\FisherInfob_{I\!, \lpnui_k}^\top&\FisherInfo_{\lpnui_k \lpnui_1}&\ldots&\FisherInfo_{\lpnui_k \lpnui_k}
	\end{array}\right)\right).
\end{equation}
It follows that the log-likelihood ratios for the marginal Gaussian    experiment with observation~$\Deltab_{\text{shift}}^{(\ell)}$ take the typical Gaussian form 
\begin{equation}\label{DlogL}
	\log \dfrac{
		\dd\prob_{\lpint ,\lpnui}
	}{
		\dd\prob_{\mfzero,0}
	}(\Deltab_{\text{shift}}^{(\ell)}) =
	\CSb_I^\top\lpint
	+ \CS_{I\!I\!,\lpnui} 
	-\frac{1}{2}\left[
	\lpint^\top\FisherInfob_{I\!,I}\lpint + 2\lpint^\top\FisherInfob_{I\!,\lpnui}  + \FisherInfo_{\lpnui\lpnui}
	\right]. 
\end{equation}

The parameters $\lpint$ and~$\lpnui$ in \eqref{margDshift} are shifting the finite-dimensional marginal Gaussian distributions of $\Deltab_{\text{\rm shift}} $ without affecting their covariance structure, whence the classical terminology, for ${\cal E}_{\text{shift}}$, of   {\it Gaussian shift experiment}.

Let us show, for the reasons described in the introduction, that the same limiting experiment  ${\cal E}_{\thetab_0, \varthetab_0}$ also can be described as a  {\it Brownian drift} experiment~${\mathcal E}_{\text{{\rm drift}}} = {\cal E}_{{\scriptstyle{\text{\rm drift;}} {(\thetab_0, \varthetab_0)}}}$, living on a ``richer'' $\sigma$-field than ${\cal E}_{\text{shift}} = {\cal E}_{{\scriptstyle{\text{\rm shift;}} {(\thetab_0, \varthetab_0)}}}$. More precisely, the distributions $\prob_{\lpint ,\lpnui}$ in~${\mathcal E}_{\text{{\rm drift}}}$ are described as those of a family
\begin{eqnarray}\nonumber
 \Deltab_{\text{\rm drift}} &=& \Deltab_{\text{\rm drift};(\thetab_0,\varthetab_0)}\\
\label{Bmotions} 
 &\coloneqq& \Big\{\!\Deltab_{{\text{\rm int}}} ^\top(u)
	\!\coloneqq\!\big(\Delta_{{\text{\rm int}},1}(u),\ldots,\Delta_{{\text{\rm int}},k}(u)\big)\!, \Deltab_{{\text{\rm nuis}}}(u)\!\coloneqq\!  \big\{\CS_{{\text{\rm nuis}},\lpnui}(u)\vert \lpnui\in\Upsilon_{\varthetab_0}\big\} \! :\!\,u\in [0,1]\!\Big\}
\end{eqnarray}
of multivariate Brownian motions defined over some appropriate (but irrelevant) probability space. 
Under $\prob_{\mfzero,0}$, these Brownian motions have zero drift and the same (full-rank) finite-dimensional covariance matrices $\FisherInfob_{\eta_1,\ldots,\eta_\ell}$ as in \eqref{margDshift}---recall that the mean and covariance structure at $u=1$ completely determine the distribution of a multivariate Brownian motion. Under $\prob_{\lpint ,\lpnui}$,  the components~$\Deltab_{{\text{\rm int}}}$ and~$\CS_{{\text{\rm nuis}}, \lpnui^\prime}$ of $\Deltab_{\text{\rm drift}}$ still are Brownian motions, with the same covariance structure  as under~$\prob_{\mfzero,0}$, but now with drifts  $\FisherInfob_{I\!, I}\lpint +\FisherInfob_{I\!,\lpnui}$ and~$\FisherInfob_{I\!, \lpnui^\prime}^\top\lpint+\FisherInfo_{\lpnui \lpnui^\prime}$, respectively.    This characterizes all finite-dimensional marginals of $\prob_{\mfzero,0}$ and $\prob_{\lpint ,\lpnui}$, hence~$\prob_{\lpint ,\lpnui}$ and $\prob_{\mfzero,0}$ themselves and their log-likelihood ratios.

We then have the following result.
\begin{prop} The limit experiments $\E_{\pint_0 ,f_0}$ described via~$\Deltab_{\text{\rm shift} ;(\thetab_0,\varthetab_0)}$ as a Gaussian shift experiment $ {\mathcal E}_{\text{\rm shift};(\thetab_0,\varthetab_0)}$ and via $\Deltab_{\text{\rm drift};(\thetab_0,\varthetab_0)}$ as a Brownian drift experiment~$ {\mathcal E}_{\text{{\rm drift}};(\thetab_0,\varthetab_0)}$ are equivalent in the sense of the Le Cam distance. 
\end{prop}
\noindent{\sc Proof.} Under $\prob_{\lpint ,\lpnui}$, the components~$\Deltab_{{\text{\rm int}}}$ and~$\CS_{{\text{\rm nuis}}, \lpnui^\prime}$ of $\Deltab_{\text{\rm drift}}$  are Brownian motions, with the same covariance structure  as in \eqref {margDshift} and drifts  $\FisherInfob_{I\!, I}\lpint +\FisherInfob_{I\!,\lpnui}$ and~$\FisherInfob_{I\!, \lpnui^\prime}^\top\lpint+\FisherInfo_{\lpnui \lpnui^\prime}$, respectively. By Girsanov's theorem, the log-likelihood ratios $\log {
		\dd\prob_{\lpint ,\lpnui}
	}/{
		\dd\prob_{\mfzero,0}
	}$ take the form 
\begin{equation}\label{BrownlogL}
	\log \dfrac{
		\dd\prob_{\lpint ,\lpnui}
	}{
		\dd\prob_{\mfzero,0}
	}(\CSdrift
	)
	=
	\CSb_I^\top(1)\lpint 
	+ \CS_{I\!I\!,\lpnui}(1)
	\mbox{}-\frac{1}{2}\left[
	\lpint^\top\FisherInfob_{I\!, I}\lpint + 2\lpint^\top\FisherInfob_{I\!,\lpnui}  + \FisherInfo_{\lpnui\lpnui}
	\right].
\end{equation}
Note that \eqref{BrownlogL} and \eqref{DlogL}, with  $\CSb_I(1)$ and $ \CS_{I\!I\!,\lpnui}(1)$ substituted for $\CSb_I$ and $\CS_{I\!I\!,\lpnui}$, coincide and that, moreover,  
$
\CSdrift(1)\coloneqq \big\{\CS_{I\!,1}(1),\ldots,\CS_{I\! ,p}(1), \{\CS_{I\!I\!,\lpnui}(1): \lpnui\in\LPnui\}\big\}
$ 
and ${\CSb}_{\text{{shift}}}$ are sharing, under $(\lpint, \lpnui)=(\mfzero, 0)$, the same finite-dimensional Gaussian distributions \eqref{Normal0}. It follows that the distributions of  
$$ \left( 
\log \dfrac{
	\dd\prob_{\lpint_1 ,\lpnui_1}
}{
	\dd\prob_{\mfzero,0}
}(\CSdrift ),\ldots ,
\log \dfrac{
	\dd\prob_{\lpint_k ,\lpnui_\ell}
}{
	\dd\prob_{\mfzero,0}
}(\CSdrift )
\right)$$ 
under $\CSdrift\sim \prob_{\mfzero,0}$, and those of 
$$ \left( 
\log \dfrac{
	\dd\prob_{\lpint_1 ,\lpnui_1}
}{
	\dd\prob_{\mfzero,0}
}(\CSshift ),\ldots ,
\log \dfrac{
	\dd\prob_{\lpint_k ,\lpnui_\ell}
}{
	\dd\prob_{\mfzero,0}
}(\CSshift)
\right)$$ 
under $\CSshift\sim \prob_{\mfzero,0}$ coincide for any $k$ and $\ell$ and any $k\ell$-tuple $(\lpint_1,\lpnui_1),\ldots,  (\lpint_k,\lpnui_\ell)$. From this, we conclude that, for all $(\thetab_0,\varthetab_0)$, the limiting Gaussian shift experiment $ {\mathcal E}_{\text{\rm shift}; (\thetab_0,\varthetab_0)} $ and the limiting Brownian drift experiment $ {\mathcal E}_{\text{\rm drift}; (\thetab_0,\varthetab_0)}$ coincide in the sense of the Le Cam distance.\hfill$\square$\medskip
 
We focus, from now on, on the Brownian drift representations ${\mathcal E}_{\text{\rm drift}; (\thetab_0,\varthetab_0)}$ of LAN local limiting experiments.

\subsection{Sufficiency and ancillarity in Brownian drift experiments}\label{Sec23} 
The concept of (bounded) completeness, in connection with sufficiency and ancillarity, is playing a major role at the foundations of statistical inference: unbiased point estimation (the so-called Lehmann-Scheff\' e Theorem), unbiased testing (similarity and Neyman $\alpha$-structure), and statistical decision in the presence of nuisance parameters. These subjects have  generated an abundant literature, some controversy, and several puzzles that remain unsettled or undecidable. Recall that the main problem, from the perspective of this paper, with finite-$n$ nuisance-ancillarity is the non-uniqueness of maximal ancillary $\sigma$-fields. Let us show that this uniqueness problem does not arise in the limiting Gaussian shift or Brownian drift experiments.

Consider the limiting  Brownian drift experiment ${\mathcal E}_{\text{{\rm drift}}} =  {\mathcal E}_{\text{{\rm drift}}; (\thetab_0,\varthetab_0)}$ as described in Section~\ref{Sec22}. Call {\it nuisance experiment} the subexperiment ${\mathcal E}_{\text{{\rm drift}}}^{\boldsymbol 0} =  {\mathcal E}_{\text{{\rm drift}}; (\thetab_0,\varthetab_0)}^{\boldsymbol 0}$  of   $ {\mathcal E}_{\text{\rm drift}}$ associated with the subfamily
$\{\prob_{\mfzero,\eta}:\, \eta\in\Upsilon_{\varthetab_0}\}$ of $\{\prob_{\taub,\eta} :\,  \taub\in{\mathbb R}^k, \eta\in\Upsilon_{\varthetab_0}\}$ and call {\it (minimal) nuisance-sufficient} and {\it (maximal) nuisance-ancillary} any $\sigma$-field or statistic which is (minimal) sufficient or (maximal) ancillary in this nuisance experiment ${\mathcal E}_{\text{\rm drift}}^{\boldsymbol 0}$. Recall that a sub-$\sigma$-field ${\cal B}_0\subseteq{\cal B}$ is called  {\it (boundedly) complete} in ${\mathcal E}_{\text{\rm drift}}^{\boldsymbol 0}$ if, denoting by $\zeta$ a (boun\-ded)~${\cal B}_0$-measurable random variable and writing ${\mathbb E}_{\mfzero,\eta}$ for expectations under $\prob_{\mfzero,\eta}$, for all~$\eta\in\Upsilon_{\varthetab_0}$,
$${\mathbb E}_{\mfzero,\eta}(\zeta )=0\ \text{ for all }\  \eta\in\Upsilon_{\varthetab_0}\quad \text{ implies }\quad \zeta =0\ \  \prob_{\mfzero,\eta}\text{-a.s.}  \text{ for all }\  \eta\in\Upsilon_{\varthetab_0}.$$ When the (full) $\sigma$-field $\cal B$ in an experiment ${\cal E}=(\Omega, {\cal B}, {\cal P})$ is (boundedly) complete, we also say that the experiment $\cal E$  is (boundedly) complete.

The sufficiency, completeness, and ancillarity properties of the nuisance subexperiment~${\mathcal E}_{\text{\rm drift}}^{\boldsymbol 0}$  of~$ {\mathcal E}_{\text{\rm drift}}$ then can be summarized as follows  (recall that $\Upsilon_{\varthetab_0}$ throughout is assumed to be separable).
\begin{prop}\label{UniqueAnc} 
In the Brownian drift experiment ${\mathcal E}_{\text{{\rm drift}}} = {\mathcal E}_{\text{{\rm drift}}; (\thetab_0,\varthetab_0)}$ with observation $\Deltab_{\text{\rm drift}} \!=\! \Deltab_{\text{\rm drift};(\thetab_0,\varthetab_0)}$ defined in \eqref{Bmotions}, let 
$\FieldSuf\coloneqq \sigma\left(\left\{\dfrac{
		\dd\prob_{{\boldsymbol 0} ,\lpnui}}{\dd\prob_{\mfzero,0}}(\Deltab_{\text{\rm drift}})	: \lpnui\in\LPnui_{\varthetab_0} \right\}\right)$.
Then
\begin{enumerate}
\item[(i)] $\FieldSuf = \sigma \left(\big\{\Delta_{{\text{\rm nuis}},\eta}(1): \lpnui\in\LPnui_{\varthetab_0}\big\} \right)$;
\item[(ii)] $\FieldSuf$ (equivalently, the collection of statistics $\big\{\Delta_{{\text{\rm nuis}},\eta}(1)  : \, \eta \in\Upsilon_{\varthetab_0}\big\}$) is nuisance-minimal sufficient (${\mathcal E}_{\text{\rm drift}}^{\boldsymbol 0}$-minimal sufficient);
\item[(iii)] $\FieldSuf$ (equivalently, the collection  of statistics 
$\big\{\Delta_{{\text{\rm nuis}},\eta}(1) : \, \eta \in\Upsilon_{\varthetab_0}\big\}$)
is boundedly complete in~${\mathcal E}_{\text{\rm drift}}^{\boldsymbol 0}$; 
\item[(iv)] $\sigma\left( \Deltab_{\text{\rm drift}}\right)$ is sufficient and  boundedly complete in ${\mathcal E}_{\text{\rm drift}}$ (equivalently, ${\mathcal E}_{\text{\rm drift}}$ is boundedly complete); 
\item[(v)] for any $\eta^\prime\in\Upsilon_{\varthetab_0}$, $\left\{\big\{B_\eta(u) \coloneqq \CS_{{\text{\rm nuis}},\lpnui}(u) -u\CS_{{\text{\rm nuis}},\lpnui}(1),\ u\in[0,1]\big\} :\, \eta \in\Upsilon_{\varthetab_0}\right\}$ is an (infinite-dimensional) Brownian bridge  under $\prob_{{\boldsymbol 0},\lpnui^\prime}$,  hence is ancillary in the nuisance experiment ${\mathcal E}_{\text{\rm drift}}^{\boldsymbol 0}$;
\item[(vi)] ${\mathcal B}^\ddagger
\coloneqq \sigma\left(\{B_{\lpnui}(u) : \lpnui\in\LPnui_{\varthetab_0},\, u\in[0,1]\}\right)$ is the unique maximal nuisance-ancillary $\sigma$-field in~${\mathcal E}_{\text{\rm drift}}$.
\end{enumerate}
\end{prop}
\noindent{\sc Proof.} 
%
(i) 
  This readily follows from 
   \eqref{BrownlogL}, which yields 
   \begin{equation}\label{Bsuff}
   \log \dfrac{
		\dd\prob_{{\boldsymbol 0} ,\lpnui}
	}{
		\dd\prob_{\mfzero,0}
	}(\CSdrift
	)
	=
	\CS_{I\!I\!,\lpnui}(1)
	-
	 \FisherInfo_{\lpnui\lpnui}
	/2\quad\text{for all $\eta\in\Upsilon_{\varthetab_0}$.}
\end{equation}
  
 (ii)    The  factori\-zation criterion and \eqref{Bsuff} imply   that  
 $\FieldSuf$ 
is sufficient in~${\mathcal E}_{\text{{\rm drift}}}^{\boldsymbol 0}$, that is,~${\mathcal E}_{\text{{\rm drift}}}$-nuisance-sufficient at $\taub = {\boldsymbol 0}$. Clearly, any $\sigma$-field sufficient in ${\mathcal E}_{\text{{\rm drift}}}^{\boldsymbol 0}$  also is sufficient in any subexperiment  of the form~${\mathcal E}_{\eta_1,\ldots,\eta_\ell}^{\boldsymbol 0}\coloneqq \big(\outcome, \Field, 
 \left\{\prob_{{\boldsymbol 0},\lpnui}:   \lpnui\in\{{\eta_1,\ldots,\eta_\ell}\} \right\} \big)$ for any  $\ell\in\mathbb N$ and any 
  $\{\eta_1,\ldots,\eta_\ell\}\subset \Upsilon_{\varthetab_0}$. Now, similarly as in (i),
$$\FieldSuf^{\eta_1,\ldots,\eta_\ell}\coloneqq \sigma\left(\left\{\dfrac{
		\dd\prob_{{\boldsymbol 0} ,\lpnui}}{\dd\prob_{\mfzero,0}}(\CSdrift)	: \eta\in\{\eta_1,\ldots,\eta_\ell\}\right\}\right) = \sigma \big(\{\Delta_{I\!I\!,\eta}(1): \eta\in\{\eta_1,\ldots,\eta_\ell\}\big) 
		,$$
	and the factorization criterion implies that $\FieldSuf^{\eta_1,\ldots,\eta_\ell}$	 is minimal sufficient  for ${\mathcal E}_{\eta_1,\ldots,\eta_\ell}^{\boldsymbol 0}$: a nuisance-sufficient $\sigma$-field for ${\mathcal E}_{\text{{\rm drift}}}^{\boldsymbol 0}$, thus, should contain   all $\sigma$-fields of the form~$\FieldSuf^{\eta_1,\ldots,\eta_\ell}$, hence should contain    $\FieldSuf$. It follows that $\FieldSuf$ is ${\mathcal E}_{\text{\rm drift}}^{\boldsymbol 0}$-minimal sufficient, as was to be shown. 

\newcommand{\FieldSub}{{\cal F}}
(iii) A result by \cite{Far62} (see Section~2 in \cite{hallin2023bounded}) establishes a connection between the bounded completeness of the $\sigma$-field ${\cal B}$ in an experiment of the form $\E\coloneqq(\Omega,{\cal B}, {\cal P}\coloneqq\{{\rm P}_{\thetab}\vert\; \thetab\in\Thetab\})$ (equivalently, the bounded completeness of $\E$ itself) and the denseness (in the sense of the~$L^1$ norm) of Radon-Nikodyn derivatives in~$L^1({\cal B}, {\mathbb P}_{{\boldsymbol 0},0})$. An extension of that result is given in Proposition~2 of~\cite{hallin2023bounded} in order to  cope with the completeness of sub-$\sigma$-fields ${\cal B}_0\subseteq~\!{\cal B}$ defining {\it restricted  experiments} of the form~$\E_{{\cal B}_0}\coloneqq~\!(\Omega,{\cal B}_0, {\cal P}\coloneqq~\!\{{\rm P}_{\thetab}\vert\; \thetab\in\Thetab\})$. The proof of bounded completeness of $\FieldSuf$ is based on that extension.

Let the sequences~$(u_k)_{k\in\rN}$  and~$(\lpnui_k)_{k\in\rN}$ be  dense in $[0,1]$ and  $\LPnui_{\varthetab_0}$, respectively (such sequences exist since $\LPnui_{\varthetab_0}$ is separable) and such that $u_i\neq u_j$, $\eta_i\neq\eta_j$ for $i\neq j$. 
 Define 
$$\FieldSub_k\coloneqq 
\sigma\left(\left\{\dfrac{
\dd\prob_{\mfzero,\lpnui_i}}{\dd\prob_{\mfzero,0}}(\CSdrift(u_j)):i,j\leq k\right\}\right)\quad k\in{\mathbb N}.$$
The sequence $\FieldSub_k$, $k\in{\mathbb N}$ is a convergent filtration in $\FieldSuf$. 
 The subexperiment $\E_{\FieldSub_k}$ obtained by restricting 
 ${\mathcal E}_{\text{{\rm drift}}}^{\boldsymbol 0}$ to the sub-$\sigma$-field  $\FieldSub_k$ is a full-rank  $k^2$-dimensional~Gaussian shift under which the~$k$-tuple  $(\CS_{{\text{\rm nuis}},\lpnui_1}(\s),\dots,\CS_{{\text{\rm nuis}},\lpnui_k}(\s))$ of Brownian motions 
   is observed at $u$ in~$\{u_1,\ldots,u_k\}$. Classical properties of exponential families, thus, imply that the $\sigma$-field 
    $\FieldSub_k$  is complete in $\E_{\FieldSub_k}$ for any $k\in{\mathbb N}$. Letting 
    $${\cal S}_{\FieldSub_k}\coloneqq\text{\rm span}\Big(\big\{{\mathbb E}_{{\boldsymbol 0},0} \big[{\rm d {\mathbb P}_{{\boldsymbol 0},\eta}}/{\rm d}{\mathbb P}_{{\boldsymbol 0},0}\vert \FieldSub_k
   \big]: \eta\in\Upsilon_{\varthetab_0}
   \big\} \Big)
    $$
    and
    $${\cal S}_{\FieldSuf}\coloneqq\text{\rm span}\Big(\big\{{\mathbb E}_{{\boldsymbol 0},0} \big[{\rm d {\mathbb P}_{{\boldsymbol 0},\eta}}/{\rm d}{\mathbb P}_{{\boldsymbol 0},0}\vert \FieldSuf
   \big]: \eta\in\Upsilon_{\varthetab_0}
   \big\} \Big)$$
    where span(...) denotes the collection of all linear combinations of $(...)$, we have ${\cal S}_{\FieldSub_k}\subseteq {\cal S}_{\FieldSuf}$ for all $k\in{\mathbb N}$.    In view of \citet[Proposition~2]{hallin2023bounded}, the completeness, for all $k$, of~$\FieldSub_k$ then  implies that $\FieldSuf$ also is complete---{\it a fortiori} boundedly complete.

(iv)  Note that (iii) does not imply (iv); nor does (iv) imply (iii).  Being the observation in~${\mathcal E}_{\text{\rm drift}}$, $\Deltab_{\text{\rm drift}}$  is trivially sufficient in ${\mathcal E}_{\text{\rm drift}}$.
 The same reasoning as in~(iii)  applies to~$\FieldSuf$ in~${\mathcal E}_{\text{\rm drift}}$. The result follows.

(v) The claim readily follows from the definition of  Brownian motion and the fact that the marginals and the covariance structure are nuisance-free.

(vi) 
As a Brownian bridge, $B_\eta$ is distribution-free under any $\prob_{{\boldsymbol 0} ,\lpnui^\prime}$.  Hence, ${\mathcal B}^\ddagger$ is ${\mathcal E}_{\text{{\rm drift}}}^{\boldsymbol 0}$-ancillary. 
More\-over,~$\sigma\left(\FieldSuf \bigcup {\mathcal B}^\ddagger\right) = {\mathcal B}$.  It thus follows from a theorem by \cite{basu1959} (see Appendix E, Corollary E.1 in the online supplement to \cite{hallin2021distribution} that~${\mathcal B}^\ddagger$ is~${\mathcal E}_{\text{{\rm drift}}}^{\boldsymbol 0}$-maximal ancillary. 
 In order to prove the much stronger property that 
 ${\mathcal B}^\ddagger$  is the {\it unique}~${\mathcal E}_{\text{{\rm drift}}}^{\boldsymbol 0}$-maximal ancillary $\sigma$-field, it is sufficient to  show that  any~${\mathcal E}_{\text{{\rm drift}}}^{\boldsymbol 0}$-ancillary statistic is
 ~${\mathcal B}^\ddagger$-measurable. 
 
For all $ \eta	\in\Upsilon_{\varthetab_0}$, consider the closed $L^2$ space $L^2_{\eta}\coloneqq L^2\left(
		\CSdrift ,
{\mathbb P}_{{\boldsymbol 0},\eta}\right)$ (equipped with the covariance inner product) of all linear combinations of the components of $\CSdrift (u)$, $u\in[0,1]$ and the limits of  the~${\mathbb P}_{{\boldsymbol 0},\eta}$-quadratic-mean convergent sequences thereof. Because the covariance structure of~$\CSdrift $ does not depend on $\eta$, $L^2_{\eta} = L^2_{0}$ irrespective of $\eta\in\Upsilon_{\varthetab_0}$. Similarly consider, for all $\eta^\prime\in\Upsilon_{\varthetab_0}$, the subspaces 
    $$L^2_{\Delta_{{{\text{\rm nuis}},\eta^\prime}}(1)}\coloneqq L^2\Big(\big\{\Delta_{{\text{\rm nuis}} ,\eta}(1)  : \, \eta \in\Upsilon_{\varthetab_0}
		\big\}, {\mathbb P}_{{\boldsymbol 0},\eta^\prime}\Big)\subset L^2_{\eta^\prime}= L^2_{0} $$ 
spanned by $\left\{\Delta_{{\text{\rm nuis}},\eta}(1)  : \, \eta \in\Upsilon_{\varthetab_0}
\right\} $ (recall from (i)  that $  \left\{\Delta_{{\text{\rm nuis}},\eta}(1)  : \, \eta \in\Upsilon_{\varthetab_0}
\right\}$  generates~$\FieldSuf$) and their orthogonal complements $L^{2\perp}_{\Delta_{{\text{\rm nuis}},\eta^\prime}(1)}$ in~$L^2_{\eta^\prime}=L^2_{0} $.
For the same reasons as above, 
    $$L^{2}_{\Delta_{{\text{\rm nuis}},\eta^\prime}(1)}\!\! =L^{2}_{\Delta_{{\text{\rm nuis}},0}(1)}\eqqcolon L^{2}_{\Delta_{{\text{\rm nuis}}}(1)}\quad\text{and}\quad L^{2\perp}_{\Delta_{{\text{\rm nuis}},\eta^\prime}(1)}\! \!=L^{2\perp}_{\Delta_{{\text{\rm nuis}},0}(1)}\eqqcolon L^{2\perp}_{\Delta_{{\text{\rm nuis}}}(1)},$$
where $L^{2}_{\Delta_{{\text{\rm nuis}}}(1)}$ and $L^{2\perp}_{\Delta_{{\text{\rm nuis}}}(1)}$ are mutually orthogonal under ${\mathbb P}_{{\boldsymbol 0},\eta}$ for any $\eta\in\Upsilon_{\varthetab_0}$. For simpli\-city, denote them as $L^0$ and $L^\perp$, respectively, and note that the elements of $L^0$ are~$\FieldSuf$-measurable.

Being jointly normal and mutually orthogonal under ${\mathbb P}_{{\boldsymbol 0},\eta}$ for any $\eta\in\Upsilon_{\varthetab_0}$, the elements of~$L^0$ and $L^\perp$ are mutually independent under ${\mathbb P}_{{\boldsymbol 0},\eta}$ for any $\eta\in\Upsilon_{\varthetab_0}$. Since, more\-over,~$\FieldSuf$ 
is (from (ii) and~(iii))~$ {\mathcal E}^0_{\text{{\rm drift}}}$-minimal sufficient and complete, Basu's first theorem (see Appendix~\ref{Basu}; note that \cite{KandT75}'s condition is trivially satisfied here) thus implies that  all elements of $L^\perp$ are ancillary in $ {\mathcal E}^0_{\text{{\rm drift}}}$. 	Projecting (in $L^2_{0} $) $\CSdrift$ on $L^{0}$ and~$L^{\perp}$, respectively, yields a decomposition of  $\CSdrift$ into a sum $\CSdrift^{0} + \CSdrift^{\perp}$ of two mutually independent  (under $\prob_{\mfzero,0}$, for all $\eta\in\Upsilon_{\varthetab_0}$) vectors, where~$\CSdrift^{\perp}$ is $ {\mathcal E}^0_{\text{{\rm drift}}}$-ancillary. 
		
Now, consider a bounded~${\mathcal E}_{\text{{\rm drift}}}^{\boldsymbol 0}$-ancillary statistic, i.e., a bounded $\CSdrift $-measurable 
		  random variable $h(\CSdrift ) = h(\CSdrift^{0} + \CSdrift^{\perp})$ the distribution of  which is the same under any $\prob_{\mfzero,\lpnui}$ as under $\prob_{\mfzero,0}$. 
		 With a slight abuse of notation,~$h(\CSdrift )$ can be written as~$h(\CSdrift^{0} , \CSdrift^{\perp})$. 
		  Denoting by $\Omega^0$ and $\Omega^\perp$, with Borel $\sigma$-fields ${\mathcal B}^0$ and ${\mathcal B}^\perp$, the sample spaces of $\CSdrift^{0}$ and~$\CSdrift^{\perp}$, 
		   consider elements $B$ of the Borel $\sigma$-field $\mathcal B$ of $\mathbb R$ such that~$h^{-1}(B) = (B^0\times B^\perp)$ for some~$B^0\in{\mathcal B}^0$ and $B^\perp\in {\mathcal B}^\perp$. 
		   %
Because $h$ is ancillary, for any such $B$,  due to the independence of~$\CSdrift^{0}$ and $\CSdrift^{\perp}$ and the $ {\mathcal E}^0_{\text{{\rm drift}}}$-ancillarity of $\CSdrift^{\perp}$, 
\begin{equation}\label{facteq}
\begin{aligned}
{\mathbb P}_{{\bf 0}, \eta}[h(\CSdrift^{0} , \CSdrift^{\perp})\in B] &= {\mathbb P}_{{\bf 0}, \eta}[(\CSdrift^{0} , \CSdrift^{\perp})\in B^0\times B^\perp] \\ 
&={\mathbb P}_{{\bf 0}, \eta}(\CSdrift^{0} \in B^0){\mathbb P}_{{\bf 0}, \eta}(\CSdrift^{\perp} \in B^\perp ) \\
&={\mathbb P}_{{\bf 0}, \eta}(\CSdrift^{0} \in B^0){\mathbb P}_{{\bf 0}, 0}(\CSdrift^{\perp} \in B^\perp )
\end{aligned}
\end{equation}
does not depend on $\eta$. This is possible only if  ${\mathbb P}_{{\bf 0}, \eta}(\CSdrift^{0} \in B^0)$ itself does not depend on~$\eta$. Since ${\mathbb P}_{{\bf 0}, \eta}$ is a shifted version of ${\mathbb P}_{{\bf 0}, 0}$, which is Gaussian, the only Borel sets $B^0$ for which~${\mathbb P}_{{\bf 0}, \eta}(\CSdrift^{0} \in B^0)$   does not depend on $\eta$ are the Borel sets $B^0$ with ${\mathbb P}_{{\bf 0}, 0}$-probability zero or one (hence, ${\mathbb P}_{{\bf 0}, \eta}$-probability zero or one for any $\eta\in{\Upsilon}_{\varthetab_0}$), for which \eqref{facteq} yields 
\begin{equation}\label{almostthere} 
{\mathbb P}_{{\bf 0}, \eta}\left[(\CSdrift^{0} , \CSdrift^{\perp})\in B^0\times B^\perp\right] = {\mathbb P}_{{\bf 0}, 0}(\CSdrift^{\perp} \in B^\perp )\indicator\big[{\mathbb P}_{{\bf 0}, \eta}\big(\CSdrift^{0} \in B^0\big) >0\big].
\end{equation}

Finally, consider an arbitrary Borel set $B$ in $\mathbb R$: $h^{-1}(B)$ can be written as a countable union~$\bigcup_{i\in{\mathbb N}}B^0_i\times B^\perp_i$ of mutually disjoint product sets   $B^0_i\times B^\perp_i$ for which \eqref{almostthere} holds and~${\mathbb P}_{{\bf 0}, \eta}(\CSdrift^{0} \in B^0)>0$. This implies that 
$$
{\mathbb P}_{{\bf 0}, \eta}\left[h(\CSdrift^{0} , \CSdrift^{\perp})\in B\right]
=\sum_{i\in{\mathbb N}} {\mathbb P}_{{\bf 0}, 0}\left(\CSdrift^{\perp} \in B^{\perp}_i\right)  
= {\mathbb P}_{{\bf 0}, 0}\left(\CSdrift^{\perp} \in\bigcup_{i\in{\mathbb N}}B^{\perp}_i\right).
$$
It follows that $h$ is  measurable with respect to  $\CSdrift^{\perp}$, hence ${\mathcal B}^\ddagger$-measurable, as was to be shown. 
\hfill$\square$\medskip

\subsection{Weak convergence of $\sigma$-fields}
We now proceed with characterizing sequences of finite-$n$ nuisance-ancillary $\sigma$-fields converging (weakly, in a sense to be made precise) to the unique maximal nuisance-ancillary $\sigma$-field ${\mathcal B}^\ddagger$ of the limiting Brownian drift experiments~${\mathcal E}_{\text{\rm drift}}$. Recall that this allows us to choose the ``best'' maximal ancillary $\sigma$-field in the sequence of (localized) experiments, namely the one that converges to the \emph{unique} maximal ancillary $\sigma$-field in the limiting Brownian drift experiment (as characterized in Proposition~\ref{UniqueAnc}).
    
Throughout this section, we denote by ${\mathcal E}\n$ a local sequence  of the form ${\mathcal E}\n_{\thetab_0, \varthetab_0}$ considered in \eqref{localE},  converging (in the Le Cam distance)  to~$\E$ of the form~$\E_{\thetab_0, \varthetab_0}$ considered in~\eqref{limitexp}, by~${\mathcal E}\n_{\boldsymbol 0}$  and ${\mathcal E}_{\boldsymbol 0}$  the corresponding nuisance subexperiments associated with specified $\taub ={\boldsymbol 0}$.

Let us first introduce the notion of $\En$-{\it weak convergence} of \emph{random variables}. Introducing this notion requires some care due to the fact that $\Upsilon=\Upsilon_{\varthetab_0}$ is infinite-dimensional. Recall that~$\dd{\law}/\dd{\lawq}$ stands for the Radon-Nikodym derivative of the part of $\law$ which is absolutely continuous with respect to $\lawq$. 

\begin{defin}\label{weakcv} {\rm We say that a sequence $\xi_n$ of ${\mathcal B}\n$-measurable random variables {\em  converges $\En$-weakly to the $\Field$-measurable random variable} $\xi$ if, for any finite subset~$\FinitePar$ of~$(\lpint,\lpnui)$ values, we have, under $\probn_{\mfzero,0} = {\rm P}\n_{\thetab_0,\varthetab_0}$,
	\begin{equation}\label{eq:convAn}
		\begin{pmatrix}			\Big[\dd\probn_{\lpint,\lpnui}/\dd\probn_{\mfzero,0}\Big]_{(\lpint,\lpnui)\in\FinitePar}\\
			\xi_n
		\end{pmatrix}
		\lto
		\begin{pmatrix}			\Big[\dd\prob_{\lpint,\lpnui}/\dd\prob_{\mfzero,0}\Big]_{(\lpint,\lpnui)\in\FinitePar}\\
			\xi
		\end{pmatrix} \quad \text{as $n\to\infty$}
	\end{equation}
where $\lto$ denotes convergence in distribution.
}\end{defin}

Note that, for $\probn_{\mfzero,0}$-a.s.\ constants $\xi_n = \xi = c$ a.s., this   reduces to  the definition of weak convergence of $\En$ to $\E$: see, e.g., Definition~9.1 in \cite{vdVaart00}. Also note that~\eqref{eq:convAn} requires the limiting experiment $\E$ to be defined on a ``rich'' enough probability space. It must allow for convergence of all maximal ancillary statistics $\xi_n$ in the sequence $\En$. This explains why we need the Brownian-drift representation of the limiting experiment; the probability space underlying the usual Gaussian shift is not large enough to represent $\xi$.

\begin{defin}\label{weakcv2}
We say that a sequence ${\mathcal A}\n$ of sub-$\sigma$-fields of ${\mathcal B}\n$  {\em  $\En$-weakly converges to the sub-$\sigma$-field ${\mathcal A}$ of ${\mathcal B}$} if, for any  ${\mathcal A}$-measurable variable $\xi$, there exists an ${\mathcal A}\n$-measurable sequence $\xi\n$ that $\En$-weakly converges to $\xi$.
\end{defin}
Note that, in case ${\mathcal A}\n$ $\En$-weakly converges to ${\mathcal A}$, any sequence ${\mathcal A_0}\n$ with ${\mathcal A}\n\subset{\mathcal A_0}\n$ also converges to ${\mathcal A}$. This is not a problem below as we will require ${\mathcal A}\n$ to be ancillary, which makes that it cannot be extended arbitrarily.

We now come to the main technical result of the paper. Recall that we wish to use the notions of sufficiency and ancillarity to guide inference. For the Brownian drift limiting experiment, Proposition~\ref{UniqueAnc} in particular characterizes the {\it unique} maximal ancillary $\sigma$-field. But, how can we use that result in the sequence of (localized) experiments of interest? Corollary~\ref{CorolCommute} below gives the answer: under suitable conditions the operations of reducing an experiment to an ancillary (or sufficient) $\sigma$-field and taking the limit of a sequence of (localized) experiments commute.

Call \emph{countably generated} a $\sigma$-field $\AField$ admitting a countable generating family~$\{A_{[i]}~\!\!:~i\in~\!\!{\mathbb N}\}$ of  uniformly bounded variables (typically, indicators). Familiar examples of countably generated $\sigma$-fields are the Borel $\sigma$-fields over ${\mathbb R}^d$, $d<\infty$. A sufficient condition  for $\En$-weak convergence to a countably generated $\sigma$-field is as follows. 
\begin{lemma}\label{lem:countablyGenerated}
Let $\AField=\sigma\left(A_{[1]},A_{[2]},\dots\right)\subseteq{\mathcal B}$ be countably generated. If each $A_{[i]}$ is the~${\cal E}^{(n)}$-weak-limit of a sequence of~${\cal A}^{(n)}$-measurable variables, the sequence  $\AFieldn \subseteq{\mathcal B}\n$ of $\sigma$-fields \emph{$\En$-weakly converges} to $\AField$ as $n\to\infty$. 
\end{lemma}
\noindent{\sc Proof.}
Let $\xi$ be some $\AField$-measurable random variable. As $\AField$ is countable generated, we can find $\sigma\left(A_{[1]},A_{[2]},\dots,A_{[m]}\right)$-measurable random variables $\xi_{[m]}$ that converge, as $m\to~\!\infty$, almost surely, to $\xi$. Moreover, each $\xi_{[m]}$ is the ${\cal E}^{(n)}$-weak-limit of a sequence of ${\cal A}^{(n)}$-measurable variables $\xi_{[m],n}$. The following argument then shows that $\xi$ is also the ${\cal E}^{(n)}$-weak-limit of~$\xi_{[m],n}$. Indeed, recall that convergence in distribution, also jointly with the likelihood ratios in~\eqref{eq:convAn}, is metrizable by, e.g., the Lévy–Prokhorov metric. Denote this metric by $\delta$. Then, given $\varepsilon>0$, first choose $m$ so large that $\delta(\xi_{[m]},\xi)<\varepsilon/2$. Subsequently, choose $n$ so large that also $\delta(\xi_{[m],n},\xi_{[m]})<\varepsilon/2$.  
\hfill$\square$\medskip 
 
We can now state our main technical result. Note that in Condition~(i) below we assume the limiting maximal ancillary $\sigma$-field ${\cal B}^{\ddagger}$ to be unique, while its finite-sample counter\-part~${\cal B}^{\ddagger (n)}$ in Condition~(ii) not necessarily is unique.
\begin{theorem}\label{MainTheorem}
Let the $\sigma$-field $\Field$ in $\E$ be countably generated. Assume that
\begin{enumerate}
\item[(i)] $\EO$ admits a minimal sufficient and boundedly complete $\sigma$-field $\FieldSuf$ and a unique~maxi\-mal ancillary $\sigma$-field ${\cal B}^{\ddagger}$ (equivalently\footnote{This equivalence is a result of Basu's third Theorem~\ref{Basu3} in Appendix~A.}, a unique $\sigma$-field ${\cal B}^{\ddagger}$ such that~$\sigma\big(\FieldSuf\cup{\cal B}^{\ddagger}\big)$ coincides with $\Field$);
\item[(ii)] $\EnO$  admits a minimal sufficient and boundedly complete $\sigma$-field~$\FieldSufn$ and a maximal ancillary $\sigma$-field ${\cal B}^{\ddagger (n)}$ (equivalently, a $\sigma$-field ${\cal B}^{\ddagger (n)}$ such that~$\sigma\big(\FieldSufn\cup {\cal B}^{\ddagger (n)}\big)={\mathcal B}\n$);
\item[(iii)] ${\cal B}^{\ddagger(n)}$ converges $\En$-weakly to~${\cal B}^{\ddagger}$ as $n\to\infty$.
\end{enumerate}
Then, for any $l_1, l_2 \in{\mathbb N}$, any~$(\taub_{k_1},\ldots,  \taub_{k_{l_1}})\in{\mathbb R}^{k\times\ell_1} $, and any $(\eta_{\ell_1},\ldots,\eta_{\ell_{l_2}})\in\Upsilon_{\varthetab_0} ^{l_2}$, under~$\mathbb{P}^{(n)}_{\mfzero,0}$, as $n\to\infty$, 
\begin{equation}\label{eqn:convergence_conditionalexpectation}
\begin{aligned} 
&\left(\Expn_{\lpint_k,\lpnui_\ell}\Bigg[\frac{\dd\probn_{\lpint_k,\lpnui_\ell}}{\dd\probn_{\mfzero,0}}  \Big\vert{\cal B}^{\ddagger (n)}\Bigg] \right)_{\scriptstyle\!\!\! (k,\ell)\in\{k_1,\ldots, k_{l_1}\}\times\{\ell_1,\ldots, \ell_{l_2}\}}\\  
& \hspace{50mm} \lto 
\left( \Exp_{\lpint,\lpnui}\Bigg[\frac{\dd\prob_{\lpint,\lpnui}}{\dd\prob_{\mfzero,0}} \Big\vert{\cal B}^{\ddagger}\Bigg]
\right)_{(k,\ell)\in\{k_1,\ldots, k_{l_1}\}\times \{\ell_1,\ldots, \ell_{l_2}\}},
\end{aligned}
\end{equation}
where $\Expn_{\lpint,\lpnui}$ and $\Exp_{\lpint,\lpnui}$ stand for expectations under $\probn_{\lpint,\lpnui}$ and $\prob_{\lpint,\lpnui}$, respectively.
\end{theorem}

\begin{cor}\label{CorolCommute}
The restriction $\big({\cal X}\n, {\cal B}^{\ddagger (n)}, {\cal P}\n \big)$ of ${\cal E}\n$ to ${\cal B}^{\ddagger (n)}$ converges weakly, in the Le Cam sense, to the restriction  $\big({\cal X}, {\cal B}^{\ddagger}, {\cal P} \big)$ to  ${\cal B}^{\ddagger}$ of ${\mathcal E}$.
\end{cor}
\noindent{\sc Proof.} 
This result follows from a combined application of Theorem~2.1 of \cite{goggin1994convergence} and Theorem~2.1 of \cite{crimaldi2005convergence}.

Specifically, we align with the notation of \cite{goggin1994convergence} by letting $n = N$, ${\rm Q}^N =~\!\probn_{\mfzero,0}$, ${\rm Q} = \prob_{\mfzero,0}$, ${\rm P}^N = \probn_{\lpint,\lpnui}$, and ${\rm P} = \prob_{\lpint,\lpnui}$. We take the unique maximal ancillary $\sigma$-field~${\cal B}^{\ddagger (n)}$ to be countably generated by their $Y^N$, and the minimal sufficient $\sigma$-field $\FieldSufn$ to be generated by their $X^N$; similarly, ${\cal B}^{\ddagger (n)}$ is generated by $Y$, and $\FieldSuf$ by $X$. Then, ${\rm P}^N \ll~\!{\rm Q}^N$ on~$\sigma(X^N, Y^N)$, and our likelihood ratio $\dd\probn_{\lpint_k,\lpnui_\ell}/\dd\probn_{\mfzero,0}$, corresponding to their Radon-Nikodym derivative $L^N$, is measurable with respect to $\sigma(X^N, Y^N)$. 

Moreover, our conditions (i)--(iii) imply that
\begin{itemize}
    \item[(a)] the joint distribution of $(X^N, Y^N)$ weakly converges to that of $(X, Y)$;
    \item[(b)] the joint distribution of $(X^N, Y^N, L^N(X^N, Y^N))$ under ${\rm Q}^N$ weakly converges to that of $(X, Y, L(X, Y))$ under $\rm Q$, where $\Exp_{\rm Q}[L] = 1$; 
    \item[(c)] under ${\rm Q}^N$, $X^N$ and $Y^N$ are independent (which follows in our case from Basu's Theorem).
\end{itemize} 

Together, these conditions allow us to apply Theorem~2.1 of \cite{goggin1994convergence} to conclude that ${\rm P} \ll {\rm Q}$ on $\sigma(X, Y)$ and $\dd {\rm P}/\dd {\rm Q} = L$. This already holds in our setting. Moreover, for any bounded continuous function $F$, the conditional expectation $\Exp^{{\rm P}^N}[F(X^N) | Y^N]$ converges in distribution to $\Exp^{\rm P}[F(X) | Y]$ as $N \to \infty$. 

This convergence corresponds to condition (b) of Theorem~2.1 of \cite{crimaldi2005convergence}, which combined with point (b) right above (their condition (a)) in turn implies that, for each bounded continuous function $G$, 
 the distribution under ${\rm P}^N$ of the conditional expectation $\Exp[G(X^N, Y^N) | Y^N]$ converges weakly to the distribution of the conditional expectation $\Exp[G(X, Y) | Y]$ under $\rm P$. In our context, this is exactly the convergence result stated in (\ref{eqn:convergence_conditionalexpectation}).\hfill$\square$\medskip

The result, in Theorem~\ref{MainTheorem}, is stated for general experiments ${\cal E}\n$ and ${\cal E}$. For ${\cal E} = {\mathcal E}_{{\text{\rm drift}}}$, we know, from   Proposition~\ref{UniqueAnc} (v), that a minimal sufficient and a boundedly complete~$\FieldSuf$  and  a unique  maximal ancillary ${\cal B}^{\ddagger}$ exist, which therefore satisfy Assumption~(i); see \cite{basu1955}. If Assumption (ii) holds for  ${\cal B}^{\ddagger (n)}$, it follows, still  from \cite{basu1955},  that ${\cal B}^{\ddagger (n)}$  is {\it maximal}  ancillary. If, moreover, ${\cal B}^{\ddagger (n)}$ satisfies Assumption~(iii), call it a {\it strongly} maximal ancillary sequence. 

Let ${\cal B}^{\ddagger (n)}$ be strongly maximal ancillary: 
the weak convergence, in Corollary~\ref{CorolCommute}, of the restriction to  ${\cal B}^{\ddagger (n)}$ of ${\mathcal E}\n_{\thetab_0,\varthetab_0}$ to the restriction to ${\cal B}^{\ddagger}$ of ${\mathcal E}_{{\text{\rm drift}}}$ has important consequences. Essentially, whenever it holds, 
\begin{enumerate}
\item[(a)] all risk functions that can be achieved via nuisance-free procedures  in the limiting experiment ${\cal E}$ are limits (pointwise in their argument $\taub$, uniformly in the collection of risk functions) of sequences of risk functions that can be achieved via finite-$n$-nuisance-free procedures  in the local sequence ${\cal E}\n_{\thetab_0,\varthetab_0}$, hence also in the global experiments~${\cal E}\n_{\scriptstyle{\text{\rm global}}}$\vspace{1mm};
\item[(b)] since the  optimal risk functions (inference about $\taub$ with nuisance $\eta$ in $\mathcal E ={\mathcal E}_{\thetab_0, f_0}$) which, depending on the inference problem under study, are  setting the values at~$(\thetab_0, \varthetab_0)$ of the semiparametric efficiency bounds belong to that collection,  semiparametric efficiency at~$(\thetab_0, \varthetab_0)$ can be achieved by sequences of strictly nuisance-free  (in the finite-$n$ experiments)  procedures;
\item[(c)] in contrast, ``traditional'' semiparametrically efficient procedures based on tangent space projections, while achieving the same efficiency bounds, are not  finite-$n$ nuisance-free, require an adequate estimation of the nuisance, etc.; see Section~\ref{semiparameff}.
%
\end{enumerate}

\section{Nuisance-ancillarity and semiparametric efficiency}\label{semiparameff}
The problem of nuisance elimi\-nation~is central in the classical theory of
semiparametric inference as formalized by \cite{BKRW}, where a Le
Cam asymptotic perspective based on the assumption that~${\cal E}\n_{\scriptstyle{\text{\rm global}}}$ is LAN with central sequence $$\Deltab\n (\thetab, \varthetab)
\coloneqq \big({\boldsymbol\Delta}_{{\text{\rm int}}}^{(n)\top}(\thetab, \varthetab), {\boldsymbol\Delta}_{{\text{\rm nuis}}}^{(n)\top}(\thetab, \varthetab) \big)^\top,\quad  
n\in{\mathbb N}$$   
and information matrix (operator)~$\Gammab (\thetab, \varthetab)$. Nuisances in that context are eliminated via {\it tangent space projections}. 

%
%
  Under local (at~$(\thetab_0,\varthetab_0)$) parameter va\-lues~$({\boldsymbol\tau},\eta)$, 
  this central sequence 
   converges in distribution to  the unique obser\-vation 
$${\boldsymbol\Delta}\coloneqq \Big({\boldsymbol\Delta}^\top_{{\text{\rm int}}}(\thetab_0,\varthetab_0), {\boldsymbol\Delta}^\top_{{\text{\rm nuis}}}(\thetab_0,\varthetab_0)\Big)^\top\sim{\mathcal N}\left({\boldsymbol\Gamma}(\thetab_0, \varthetab_0)\Big(\begin{array}{c}\taub\vspace{-.5mm} \\ \eta\end{array}\Big),  {\boldsymbol\Gamma}({\thetab_0},{\varthetab_0})
\right)$$
of a limiting Gaussian shift experiment with location ${\boldsymbol\Gamma}(\thetab_0, \varthetab_0)\Big(\begin{array}{c}\taub\vspace{-1mm} \\ \eta\end{array}\Big)\vspace{-1mm}$ and  covariance matrix~${\boldsymbol\Gamma}({\thetab_0},{\varthetab_0})\coloneqq \left(\begin{array}{cc} \Gammab_{\thetab\thetab} (\thetab_0 , \varthetab_0) &\Gammab_{\thetab\varthetab} (\thetab_0 , \varthetab_0)  \\
\Gammab^{\top}_{\thetab\varthetab} (\thetab _0, \varthetab_0)  & \Gammab_{\varthetab\varthetab}  (\thetab_0 , \varthetab_0)\end{array}\right)$\vspace{1mm}.

The classical nuisance-elimination approach in semiparametric inference considers the restriction of that limiting experiment obtained by  projecting~$ \Deltab_{{\text{\rm int}}}(\thetab_0,\varthetab_0)$  on the space ortho\-gonal  to ${\boldsymbol\Delta}_{{\text{\rm nuis}}}(\thetab_0,\varthetab_0)$ in the L$_2$ metric induced by ${\boldsymbol\Gamma}({\thetab_0},{\varthetab_0})$---the so-called {\it tangent space projection}
$${\boldsymbol\Delta}_{{\text{\rm int}}}^*(\thetab_0,\varthetab_0)\coloneqq {\boldsymbol\Delta}_{{\text{\rm int}}}(\thetab_0,\varthetab_0) - {\boldsymbol\Gamma}_{\thetab\varthetab}(\thetab_0, \varthetab_0)({\boldsymbol\Gamma}_{\varthetab\varthetab}(\thetab_0, \varthetab_0))^{-1}{\boldsymbol\Delta}_{{\text{\rm nuis}}}(\thetab_0,\varthetab_0).$$ 
It is easy to see that, under local  parameter value   
$({\boldsymbol\tau},\eta)$, 
\begin{equation}\label{Delta*N}
{\boldsymbol\Delta}_{{\text{\rm int}}}^*\sim{\mathcal N}\Big( ({\boldsymbol\Gamma}_{\thetab\thetab} - {\boldsymbol\Gamma}_{\thetab\varthetab}{\boldsymbol\Gamma}_{\varthetab\varthetab}^{-1}{\boldsymbol\Gamma}^{\top}_{\thetab\varthetab})\taub , {\boldsymbol\Gamma}_{\thetab\thetab} - {\boldsymbol\Gamma}_{\thetab\varthetab}{\boldsymbol\Gamma}_{\varthetab\varthetab}^{-1}{\boldsymbol\Gamma}^{\top}_{\thetab\varthetab}
\Big)
\end{equation}
 where~${\boldsymbol\Gamma}_{\thetab\thetab}= {\boldsymbol\Gamma}_{\thetab\thetab}({\thetab_0},{\varthetab_0})$, ${\boldsymbol\Gamma}_{\thetab\varthetab} = {\boldsymbol\Gamma}_{\thetab\varthetab}(\thetab_0,\varthetab_0)$, etc. That distribution does not depend on the local nuisance  parameter  $\eta$ and settles the semiparametric     efficiency bounds---the best possible asymptotic performance (at $(\thetab_0, \varthetab_0)$)  for inference on $\thetab$ in the presence of the nuisance~$\varthetab$. Assuming the existence of an estimator $\widehat\varthetab\n$ of $\varthetab$ such that 
\begin{equation}
{\boldsymbol\Delta}_{{\text{\rm int}}}^{(n)*}\coloneqq {\boldsymbol\Delta}_{{\text{\rm int}}}\n({\thetab_0},\widehat\varthetab\n) - {\boldsymbol\Gamma}\n_{\thetab\varthetab}({\thetab_0},\widehat\varthetab\n)({\boldsymbol\Gamma}_{\varthetab\varthetab}\n ({\thetab_0},\widehat\varthetab\n))^{-1}{\boldsymbol\Delta}_{{\text{\rm nuis}}}\n({\thetab_0},\widehat{\varthetab}\n)
\label{tgtproj}\end{equation}
converges in distribution  as $n\to\infty$, under local (at $(\thetab_0,\varthetab_0)$) parameter values $(\taub , \eta)$, to~${\boldsymbol\Delta}_{{\text{\rm int}}}^{*}$, semiparametrically efficient (at $({\thetab_0},{\varthetab_0})$) inference on $\thetab$ then can be based on the Gaussian shift~\eqref{Delta*N} model for~${\boldsymbol\Delta}_{{\text{\rm int}}}^{(n)*}$. 

This ${\boldsymbol\Delta}_{{\text{\rm int}}}^{(n)*}$, however, (i) is not guaranteed to be ${\cal E}\n_{\scriptstyle{\text{\rm global}}}$-nuisance-ancillary at $\thetab_0$ for finite $n$; moreover, (ii)~it is (up to multiplication by a constant and an additive $o_{\rm P}(1)$ as~$n\to\infty$ under ${\rm P}\n_{\thetab_0,\varthetab_0}$ term) the only random vector the tangent space projection  matrix~${\bf I} -  {\boldsymbol\Gamma}_{\thetab\varthetab}(\thetab_0,\varthetab_0)({\boldsymbol\Gamma}_{\varthetab\varthetab}(\thetab_0,\varthetab_0))^{-1}$ is turning into an asymptotically nuisance-ancillary one: tangent space projections, thus, are failing to reconstruct the entire maximal ancillary  $\sigma$-field~${\mathcal B}^\dagger_{\thetab_0}$; finally, (iii) the estimation of~$\varthetab$, which, in practice, is infinite-dimensional,  may be delicate,  and the convergence in distribution of~${\boldsymbol\Delta}_{{\text{\rm int}}}^{(n)*}$ to ${\boldsymbol\Delta}_{{\text{\rm int}}}^{*}$ quite slow, and highly nonuniform in $\varthetab$.

In sharp contrast with the tangent space projection \eqref{tgtproj} of~${\boldsymbol\Delta}_{{\text{\rm int}}}\n({\thetab_0},\widehat\varthetab\n) $, conditioning on the strongly maximal nuisance-ancillary $\sigma$-field~${\mathcal B}^{(n)\ddagger}_{\thetab_0}$, when it exists, offers substantial advantages:
\begin{enumerate}  
\item[(i)]  ${\rm E}\big[{\bf T}\vert{\mathcal B}^{(n)\ddagger}_{\thetab_0} \big]$ is ${\cal E}\n_{\scriptstyle{\text{\rm global}}}$-nuisance-ancillary at $\thetab_0$ for any $n$ but also  for any ${\cal B}\n$-measurable random vector $\bf T$---among which~${\bf T}={\boldsymbol\Delta}_{{\text{\rm int}}}\n({\thetab}_0,\varthetab) $ for any (possibly misspe\-ci\-fied)~$\varthetab\in{\cal F}$ and~${\bf T}={\boldsymbol\Delta}_{{\text{\rm int}}}\n({\thetab}_0,\hat\varthetab\n) $; 
\item[(ii)] while the estimation of the actual nuisance $\varthetab_0$  is not compulsory, ${\rm E}[{\boldsymbol\Delta}_{{\text{\rm int}}}\n({\thetab_0},\widehat\varthetab\n)\vert {\mathcal B}^{(n)\ddagger}_{\thetab_0}] $, where  $\widehat\varthetab\n$ denotes the type of estimator considered in tangent space projection, remains  finite-$n$-nuisance-free  and asymptotically equivalent, under ${\rm P}\n_{\thetab_0,\varthetab_0}$, to ${\boldsymbol\Delta}_{{\text{\rm int}}}^{*}(\thetab_0,\varthetab_0)$, hence semiparametrically efficient irrespective of the actual value $\varthetab_0$ of the nuisance~$\varthetab$. 
\end{enumerate}

This implies that semiparametrically efficient (either at $(\thetab_0,\varthetab)$ for some given $\varthetab$ or at~$(\thetab_0,\varthetab)$ for any~$\varthetab$ for which ${\boldsymbol\Delta}_{{\text{\rm int}}}\n({\thetab_0},\widehat\varthetab\n)$ under parameter value $(\thetab_0,\varthetab)$ converges in distribution to~${\boldsymbol\Delta}_{{\text{\rm int}}}({\thetab_0},\varthetab)$) inference can be based on (expectations taken under ${\rm P}\n_{\thetab_0,\varthetab_0}$) $${\boldsymbol\Delta}_{{\text{\rm int}}}^{(n)\ddagger}\coloneqq {\rm E}\big[{\boldsymbol\Delta}_{{\text{\rm int}}}\n({\thetab}_0,\varthetab)\vert{\mathcal B}^{(n)\ddagger}_{\thetab_0} \big]\quad\text{or}\quad {\boldsymbol\Delta}_{{\text{\rm int}}}^{(n)\ddagger}\coloneqq {\rm E}\big[{\boldsymbol\Delta}_{{\text{\rm int}}}\n({\thetab}_0,\widehat\varthetab\n)\vert{\mathcal B}^{(n)\ddagger}_{\thetab_0} \big] ,$$ respectively, which are finite-$n$ nuisance-ancillary.

The above-mentioned advantages of finite-$n$ nuisance-ancillary inference are well-known. However, as explained in the introduction, a {\it unique} maximal finite-$n$ nuisance-ancillary\linebreak $\sigma$-field generally does not exist, for instance in Example~\ref{ex1}. It is this non-uniqueness that we resolve by considering exactly those finite-$n$ nuisance-ancillary statistics that converge, in the sense of Definition~\ref{eq:convAn}, to the {\it unique} corresponding nuisance ancillary statistics in the limit experiment. The next section operationalizes this idea for Example~\ref{ex1}.

\section{Application: unspecified density models and the semiparametric efficiency of residual center-outward ranks and signs}\label{Sec4}

As an important application of Theorem~\ref{MainTheorem}, consider the unspecified density model ${\mathcal E}\n_{\text{\rm global}}$  \eqref{unspecifiedfn} of Example~\ref{ex1}.  Recall that, when  $d>1$, for given $n$ and~$\thetab_0$,  many distinct maximal nuisance-ancillary at $\thetab_0$ $\sigma$-fields are available there---gene\-rated, for instance, by componentwise residual ranks. Below, we show that the $\sigma$-field ${\mathcal B}^{(n)\ddagger}_{\thetab_0}$ gene\-rated by the measure-transportation-based ranks and signs defined as {\it center-outward ranks} and {\it signs} in \cite{hallin2021distribution}, as {\it vector ranks} in \cite{Chernoetal}, is strongly maximal nuisance-ancillary at $\thetab_0$. 
 
\subsection{Center-outward ranks and signs} 
Let $\bf Z$, with distribution $\rm P$,  be Lebesgue-absolutely continuous over~$({\mathbb R}^d, {\mathcal B}^d)$, with density $f$. A famous result by McCann (1995) implies that there exists a $\rm P$-a.s.\ unique gradient of a convex function ${\bf F}_\pm$ mapping ${\mathbb R}^d$ to the open  unit ball ${\mathbb S}_d$ in ${\mathbb R}^d$ such that $\text{if } {\bf Z}\sim~\!{\rm P}$, then~${\bf F}_\pm ({\bf Z})\sim~\!{\rm U}_d$---that is,  in the convenient notation of measure transportation,  such that~${\bf F}_\pm \# {\rm P} =~\!{\rm U}_d$  (${\bf F}_\pm$ {\it pushes $ {\rm P}$ forward to ${\rm U}_d$}) where~$ {\rm U}_d$ is the spherical uniform over  ${\mathbb S}_d$. Under mild conditions, this ${\bf F}_\pm $ is a homeomorphism---see \cite{figalli2018continuity, Alberto20, Alberto24} for precise statements and details.

 
   \cite{hallin2021distribution} call ${\bf F}_\pm$ the  {\it center-outward distribution function}  of $\rm P$ (of $\bf Z$). For~$d=1$, denoting by $F$ the traditional univariate distribution function, ${\bf F}_\pm$ reduces to $2F - 1$.  
   
 Let ${\bf Z}_1\n, \ldots, {\bf Z}_n\n$ be i.i.d.\ with distribution $\rm P$ (joint distribution ${\rm P}^n$).   \cite{hallin2021distribution}   
%
%
%
 define the empi\-rical counterpart ${\bf F}_\pm\n$ of ${\bf F}_\pm$ as the minimizer of the sum~$\sum _{i=1}^n 
\left\Vert {\bf Z}_i\n -  {\scriptstyle{\mathfrak G}}\n_{\pi (i)}
\right\Vert ^2 
$ among all pairings $({\bf Z}_i\n, {\scriptstyle{\mathfrak G}}\n_{\pi (i)})$ ($\pi$ a permutation of $\{1,\ldots,n\}$) between 
 $\{{\bf Z}_1\n, \ldots, {\bf Z}_n\n\}$ and a grid~${\mathfrak G}\n\! \coloneqq 
 \{
{\scriptstyle{\mathfrak G}}\n_1, \ldots, {\scriptstyle{\mathfrak G}}\n_n
\}$, 
and show that, provided that the empirical distribution over~${\mathfrak G}\n$ converges  to ${\rm U}_d$ as $n\to\infty$, 
$$\max_{1\leq i\leq n} \left\Vert{\bf F}_\pm\n ({\bf Z}\n_i) - {\bf F}_\pm ({\bf Z}\n_i)
\right\Vert \longrightarrow 0 \quad  \text{P-a.s.\ as~} n\to\infty 
$$ (a Glivenko-Cantelli consistency property). 
\begin{center}
\begin{figure}[]\hspace{10mm}
\includegraphics[trim=10mm 15mm 10mm 15mm, clip,width=6.6cm, height=6.6cm]{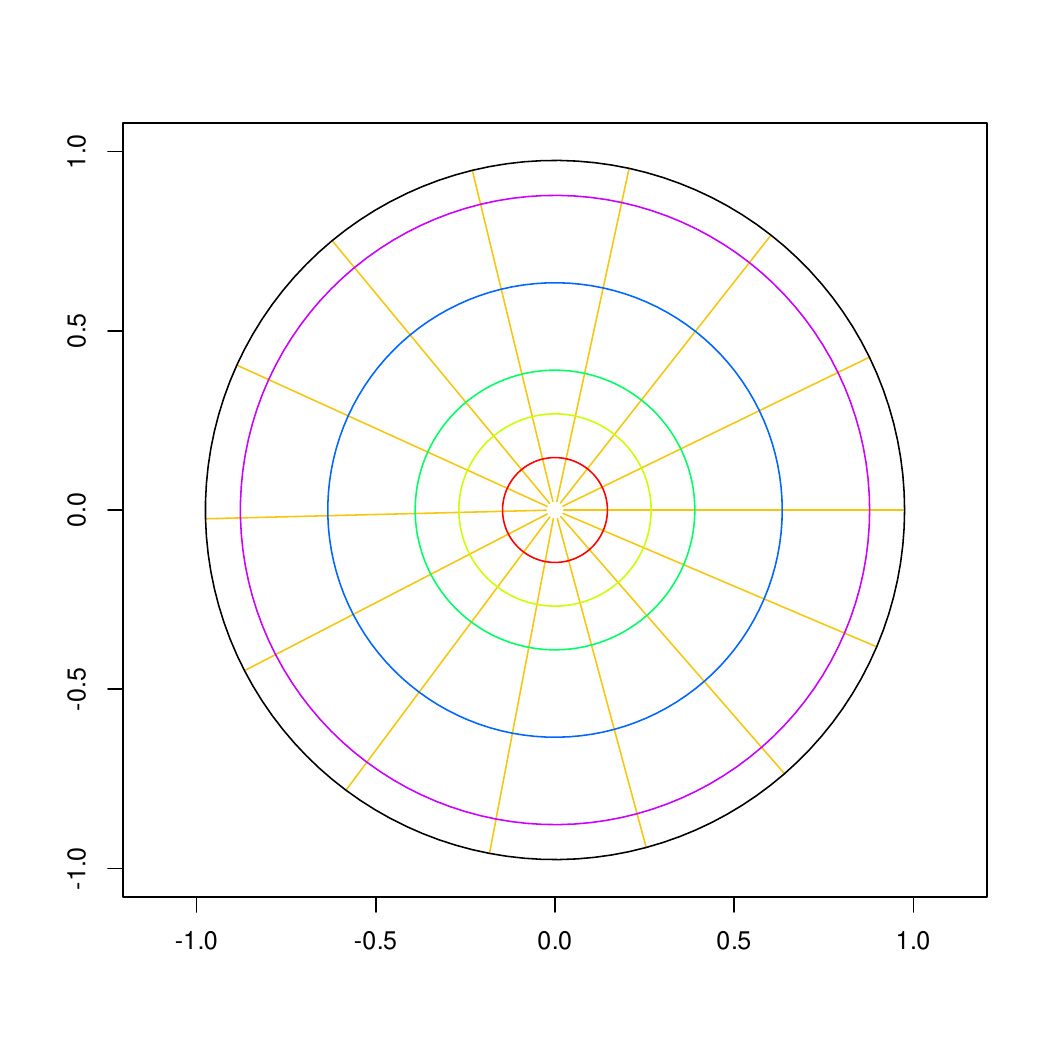}\hspace{5mm} 
\includegraphics[trim=10mm 15mm 10mm 15mm, clip,width=6.6cm, height=6.6cm]{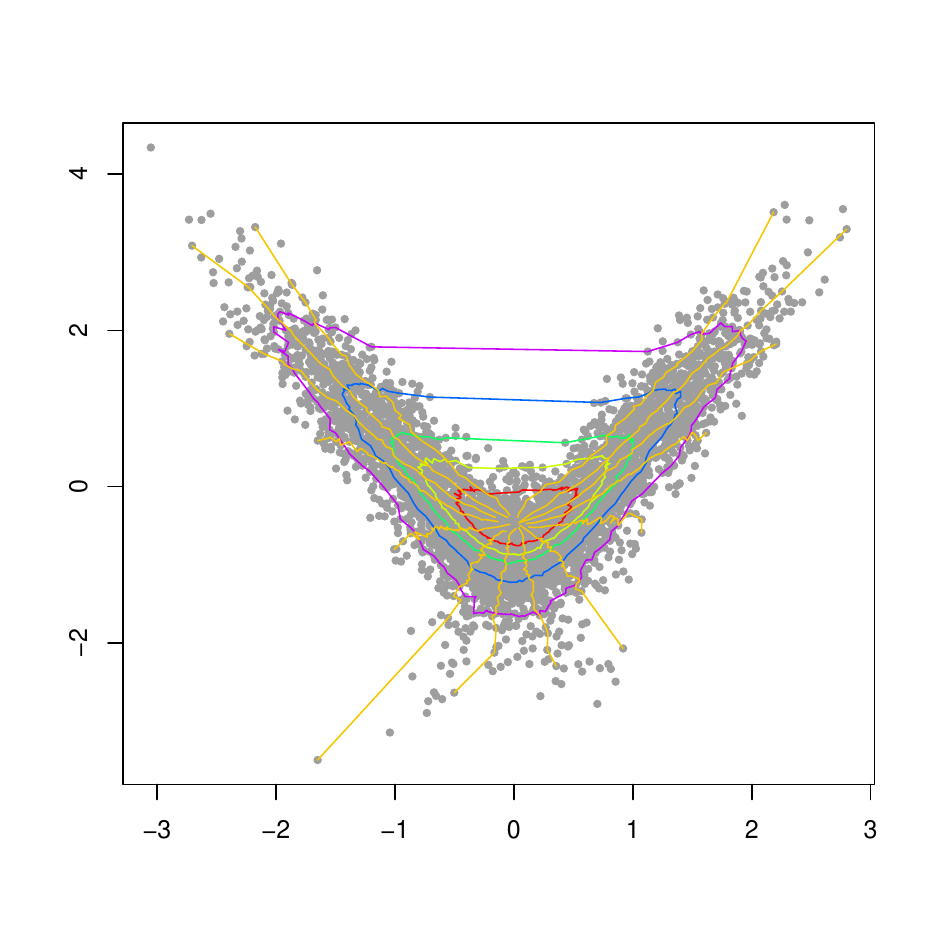} \\ \centering 
\definecolor{color1}{HTML}{FF0000}
\definecolor{color2}{HTML}{CCFF00}
\definecolor{color3}{HTML}{00FF66}
\definecolor{color4}{HTML}{0066FF}
\definecolor{color5}{HTML}{CC00FF}
\definecolor{color6}{HTML}{000000}
\fbox{ \scriptsize $\tau =$  \textcolor{color1}{\textemdash} 0.146 \quad
 \textcolor{color2}{\textemdash} 0.268\quad
 \textcolor{color3}{\textemdash} 0.390 \quad
 \textcolor{color4}{\textemdash} 0.634 \quad
 \textcolor{color5}{\textemdash} 0.878\quad
 \textcolor{color6}{\textemdash} 1.000\quad
 }\caption{The grid $\mathfrak G^{(n)}$ for $d=2$ (left-hand panel) and the empirical quantile contours (right-hand panel) of order $\tau =$ 0.146, 0.268, 0.390, 0.634, and~0.878 for a sample of size $n=5,000$ from a mixture of three bivariate Gaussians, with $n_R=40$ and $n_S=125$. The smooth interpolation of ${\bf F}^{(n)}_\pm$ described in Section~3 of \citet{hallin2021distribution} was implemented; interpolated sign curves are shown in yellow.}
\label{fig:banana}
\end{figure}
\end{center}
Since the spherical uniform is the product of a uniform over the unit sphere and a uniform over the distance to the origin, imposing a similar product structure on   ${\mathfrak G}\n$ is quite natural. Factorizing $n$ into  $n=n_Rn_S + n_0$ with $n_0<\min (n_R, n_S)$, define ${\mathfrak G}\n$ as the intersection between a regular array of~$n_S$ radii\footnote{In Figure~\ref{fig:banana}, these unit vectors are equispaced on the unit sphere $\mathcal S_{d-1}= \mathcal S_1$; this cannot be achieved for~$d>2$ and, more generally, they can be chosen as i.i.d.\ uniform over $\mathcal S_{d-1}$ (and, of course, independent of the data).} of ${\mathbb S}_d$ and a collection of spheres centered at the origin, with radii~$\frac{1}{n_R+1},\ldots,\frac{n_R}{n_R+1}$, along with $n_0$ copies of the origin; see the left-hand panel of Figure~\ref{fig:banana}. \cite{hallin2021distribution} then define the {\it center-outward rank} and  {\it sign} of~${\bf Z}_i\n$ as 
\begin{equation}\label {c-oRandS}
R\n_i\coloneqq (n_R + 1)\Vert {\bf F}_\pm\n({\bf Z}_i\n)
\Vert\quad\text{ and }\quad{\bf S}\n_i\coloneqq  {{\bf F}_\pm\n({\bf Z}_i\n)}/{\Vert {\bf F}_\pm\n({\bf Z}_i\n)
\Vert},
\end{equation}
respectively, and establish (see their Proposition~2.5) distribution-freeness, (non-unique) maximal ancillarity, and---possibly after a simple tie-breaking procedure in case $n_0>1$---their mutual finite-$n$ independence. Figure~\ref{fig:banana} provides a numerical example where ${\bf F}^{(n)}_\pm$, moreover, has been interpolated (within the class of gradients of convex functions) as explained in Section~3 of \citet{hallin2021distribution}. Quantile contours (along which $\Vert {\bf F}^{(n)}_\pm\Vert$ remains constant) and sign curves (along which ${\bf F}^{(n)}_\pm / \Vert {\bf F}^{(n)}_\pm \Vert$ remains constant) are shown in the right-hand panel. We refer to \cite{HallinSurvey} for a review of  the statistical applications of these concepts.

\subsection{Strong maximal nuisance-ancillarity of center-outward ranks and signs}
Assume that LAN holds for the unspecified density experiment ${\mathcal E}\n_{\text{\rm global}}$ in~\eqref{unspecifiedfn} of Example~\ref{ex1},  with  local experiments with global parameters $\pint^{(n)} =\pint_0 + n^{-1/2}\lpint$ and~$f\n 
=(1+n^{-1/2}\lpnui)\pnui_0$ and local parameters $\taub,\, \eta$ where~$f_0: {\bf z}\mapsto f_0({\bf z})$ is a density in $\cal F$ and~$\lpnui : {\bf z}\in\mbR^d\mapsto \lpnui({\bf z})\in\mbR$ belongs to the tangent space for location shifts
\begin{equation*}
	\LPnui_{\pnui_0}
		\!\coloneqq\!
	\left\{\!
	\lpnui {\,\bigg\vert\,}
	\!\!\int\!\! \lpnui({\bf z})\pnui_0({\bf z})\dd\mu_d = 0, \!\int\! \!\lpnui^2({\bf z})\pnui_0({\bf z})\dd\mu_d <\infty\!
	\right\}\!.
\end{equation*}
Observe that $\LPnui_{\pnui_0}$ is countably generated. Recall (from Section~\ref{Sec23} and Proposition~\ref{UniqueAnc},  with an obvious  adaptation of the notation) that, in the limiting Brownian drift experiment,
\begin{enumerate}
\item[--]  for any $\eta^\prime\in\Upsilon_{\pnui_0}$, 
   $B_\eta(u) \coloneqq \CS_{\text{\rm nuis},\lpnui}(u) -u\CS_{\text{\rm nuis},\lpnui}(1),$ $u\in[0,1], \ \eta\in\Upsilon_{\pnui_0}$   is  a    Brownian bridge under $\prob_{{\boldsymbol 0} ,\lpnui^\prime}$: therefore, $B_\eta(u)$, $u\in[0,1]$ is  normal, with mean zero and vari\-ance~$u\!\int\! \lpnui^2({\bf z})
   \dd\mu_d\eqqcolon\!\FisherInfo_{\lpnui,\lpnui}$, irrespective of  $f_0$ and  $\eta^\prime$;  
  \item[--]   ${\mathcal B}^\ddagger_{\thetab_0} \coloneqq \sigma\!\left( \{B_\eta (u): \lpnui\in \Upsilon_{f_0} ,\ u\in[0,1]\}\right)$ is the unique maximal  nuisance-ancillary\linebreak 
 $\sigma$-field in the limiting experiment~${\mathcal E}_{\text{\rm drift}}$.
\end{enumerate}

Denote by $\CODFn$ the empirical center-outward distribution function computed from the resi\-duals~${\bf Z}^{(n)}_1(\pint_0),\ldots ,{\bf Z}^{(n)}_n(\pint_0)$ and   let ${\mathcal B}^{\ddagger(n)}_{\thetab_0}\coloneqq \sigma\big(\CODFn({\bf Z}^{(n)}_1(\pint_0)),\ldots,\CODFn({\bf Z}^{(n)}_n(\pint_0))\big)$ (the $\sigma$-field generated by the center-outward ranks and signs of the residuals ${\bf Z}\n_i(\thetab_0)$, we thus have the following result.

\begin{prop} 
Under the assumptions made\footnote{This includes the assumption that    ${\Upsilon}_{f_0}$, as defined above, is separable.}, (i) the sequence $$\CODFn({\bf Z}^{(n)}_1(\pint_0)),\ldots,\CODFn({\bf Z}^{(n)}_n(\pint_0)),\quad n\in{\mathbb N}$$ asymptotically $\En$-weakly generates  ${\mathcal B}^\ddagger_{\thetab_0}
\coloneqq \sigma\!\left( \{B_\eta (u): \lpnui\in \Upsilon_{f_0} ,\ u\in[0,1]\}\right)$; (ii) the sequence ${\mathcal B}^{\ddagger(n)}_{\thetab_0}$ of $\sigma$-fields is strongly maximal ancillary.
\end{prop}
\noindent{\sc Proof.} The proof proceeds in several steps.

First, observe that ${\mathcal B}^\ddagger_{\thetab_0}$ is countably generated by~$\left\{ \Bnui (u): \lpnui\in  \Upsilon^0_{f_0},\ u \in {\cal U}_0\right\}$ where~$ \Upsilon^0_{f_0}$ and ${\cal U}_0$ denote dense subsets of $\Upsilon_{f_0}$ and $[0,1]$, respectively.   

Next, let us show that, for any $\lpnui\in \Upsilon_{f_0}$ and $u\in [0,1]$, $B_\eta (u)$ is the  limit in distribution, under ${\mathbb P}\n_{{\boldsymbol 0},0}={\rm P}_{\thetab_0,f_0}\n$, of the  sequence
\begin{equation}\label{2-samplestat}
    \tenq{T}_{\lpnui ;u}\n\coloneqq	n^{-1/2}\sum_{i=1}^{\lfloor un\rfloor}\lpnui\left(\CODFn({\bf Z}^{(n)}_i(\pint_0))\right)\quad n\in{\mathbb N}
\end{equation}
of $\big(\CODFn({\bf Z}^{(n)}_1(\pint_0)),\ldots,\CODFn({\bf Z}^{(n)}_n(\pint_0))\big)$-measurable statistics. For any $n$, $\eta$, and~$u\in[0,1]$, indeed, $\tenq{T}\n_{\lpnui ;u}\vspace{1mm}$ is a center-outward rank statistic for the two-sample problem, of the form considered in  \cite{hallin2021efficient} in their Proposition~3.1, with score function~$\lpnui$ and sample sizes~$\lfloor un\rfloor$ and $n-\lfloor un\rfloor$), respectively. It follows from the H\'{a}jek representation and asymptotic normality results in that Proposition~3.1 (note that $\eta$ has been assumed to be square-integrable, hence satisfies the required regularity assumptions) that, after due centering, $\tenq{T}\n_{\lpnui ;u} $ under ${\mathbb P}\n_{{\boldsymbol 0},0}={\rm P}_{\thetab_0,\varthetab_0}\n$ is asymptotically normal as $n\to\infty$, with mean zero and variance~$u\int \lpnui^2({\bf z})f_0({\bf z})\dd\mu_d=u\FisherInfo_{\lpnui,\lpnui}$.
	
The same technique, combined with the classical Cram\'{e}r-Wold device applied to the joint distribution of $\tenq{\bf T}\n\coloneqq \left(\tenq{T}\n_{\lpnui_1 ;u_1},\ldots,\tenq{T}\n_{\lpnui_1 ;u_{\ell}},\tenq{T}\n_{\lpnui_2 ;u_1},\ldots,\tenq{T}\n_{\lpnui_2 ;u_{\ell}},\ldots, \tenq{T}\n_{\lpnui_k ;u_1},\ldots, \tenq{T}\n_{\lpnui_k ;u_{\ell}}\right)$ for any finite $k$-tuple  $\lpnui_1,\ldots,\lpnui_k$ in $\Upsilon_{f_0}$ and any finite $\ell$-tuple $u_1,\ldots,u_{\ell}$ in $[0,1]$ yields the normal weak limits, with mean zero and covariance matrix
$$\Big(
u_{j_1}u_{j_2}\int\eta_{i_1}({\bf z})\eta_{i_2}({\bf z}){\rm d}\mu_d 
\Big)_{(i_1, i_2)\in\{1,\ldots,k\}; (j_1,j_2) \in\{1,\ldots,\ell\}},$$
of the distributions of $\tenq{\bf T}\n$ under ${\mathbb P}\n_{{\boldsymbol 0},0}={\rm P}_{\thetab_0,f_0}\n$ (which does not depend on $f_0$).

Minimal sufficiency, completeness, distribution-freeness and Basu's classical theorem then imply the independence, under $\probn_{\mfzero,0}$, of $\tenq{\bf T}\n$ and $\dd\probn_{\lpint,\lpnui}/\dd\probn_{\mfzero,0}$; independence also holds in the limiting experiment between the minimal sufficient  likelihood ratios and the ancillary Brownian bridges. Joint convergence, in the Definition~\ref{eq:convAn} of ${\cal E}\n$-weak convergence, therefore, follows from marginal convergence in distribution in \eqref{eq:convAn}. 

The desired ${\cal E}\n$-weak convergence result finally follows from Theorem~\ref{MainTheorem}.\hfill$\square$.\medskip

This has far-reaching consequences in semiparametric inference:
\begin{enumerate}
\item[(i)] semiparametric efficiency bounds at $(\thetab , f) = (\thetab_0 , f)$ can be reached by ${\mathcal B}^\ddagger_{\thetab_0}$-measurable  procedures, that is,   by strictly nuisance-free (distribution-free) procedures based on the center-outward ranks and signs of the residuals ${\bf Z}^{(n)}_1(\pint_0),\ldots, {\bf Z}^{(n)}_n(\pint_0)$; in parti\-cular, semiparametrically efficient (at  $(\thetab , f) = (\thetab_0 , f_0)$) inference for $\thetab$ can be based on~${\rm E}\big[ 
{\boldsymbol\Delta}_{{\text{\rm int}}}\n({\thetab_0},f_0)
\vert\, {\mathcal B}^{\ddagger(n)}_{\thetab_0}
\big]$ (expectation taken under ${\mathbb P}\n_{{\boldsymbol 0},0} = {\rm P}\n_{\thetab_0, f_0}$)---this expectation does not depend on $f_0$; neither does its  distribution under ${\mathbb P}\n_{{\boldsymbol 0},0} = {\rm P}\n_{\thetab_0, f_0}$; 
\item[(ii)] these ${\mathcal B}^{\ddagger(n)}_{\thetab_0}$-measurable procedures do not require $f_0$ to be the actual density, which needs not be estimated;
\item[(iii)] if, however, the actual density is consistently estimated (estimator $\hat{f}\n$), the distribution of~${\rm E}\big[ 
{\boldsymbol\Delta}_{{\text{\rm int}}}\n({\thetab_0},\hat{f}\n)
\vert\, {\mathcal B}^{\ddagger(n)}_{\thetab_0}
\big]$, where the estimator $\hat{f}^{(n)}$ only depends on the order statistic of the residuals (a very natural assumption), remains conditionally (on that order statistic) nuisance-free and yields uniformly semiparametrically efficient distribution-free testing procedures;
\item[(iv)] the components of the central sequence ${\boldsymbol\Delta}^{(n)}_{\text{\rm int}}({\boldsymbol\theta}_0,f_0)$ typically are linear combinations of variables of the form  
\begin{align*}
{\bf T}^{(n)} 
&\coloneqq \Big(
\sum_{i=1}^n (c_i^{(n)}-\bar{c}^{(n)})^2
\Big)^{-1/2}
\sum_{i=1}^n (c_i^{(n)}-\bar{c}^{(n)}) {\mathbf J}_f({\bf F}_\pm ({\bf Z}_i^{(n)})).
\end{align*}
 where the $(c^{(n)}_i - \bar{c}^{(n)})$'s, $i=1,\ldots, n$ are centered constants characterizing the problem under study (e.g., covariates: see, for instance, Proposition 4.1 in \cite{hallin2021efficient}) and ${\mathbf J}_f : {\mathbb S}_d \to{\mathbb R}^d$ is some continuous score function (e.g., the location score for $f$).  Proposition 3.1 (Ibid.) then establishes, under mild assumptions on the $c^{(n)}_i$'s, $f$, and ${\mathbf J}_f$,  the asymptotic equivalence as $n\to\infty$, under ${\mathbb P}^{(n)}_{{\boldsymbol 0},0}={\rm P}^{(n)}_{{\boldsymbol\theta}_0, {\boldsymbol\vartheta}_0}$,  of ${\bf T}^{(n)} $, its  
 {\it exact-score} version 
\begin{align*}
   \tenq{\bf T}^{(n)}_e &\coloneqq   
 \Big(
\sum_{i=1}^n (c_i^{(n)}-\bar{c}^{(n)})^2
\Big)^{-1/2}
    \sum_{i=1}^n
     (c_i^{(n)}-\bar{c}^{(n)}) 
      \E\left[ 
{\mathbf J}_f({\bf F}_\pm({\bf Z}_i^{(n)})) \,\big\vert\, {\bf F}_\pm^{(n)}({\bf Z}_i^{(n)})
 \right], 
\end{align*}
 and its much simpler {\it approximate-score} version
\begin{align*} 
  \tenq{\bf T}^{(n)}_a 
 &\coloneqq 
  \Big(
\sum_{i=1}^n (c_i^{(n)}-\bar{c}^{(n)})^2
\Big)^{-1/2}
\sum_{i=1}^n (c_i^{(n)}-\bar{c}^{(n)}) {\mathbf J}_f({\bf F}^{(n)}_\pm ({\bf Z}_i^{(n)})).
\end{align*}
This {\it asymptotic representation} result (still, in H\'ajek's terminology) allows for defining rank-based, hence distribution-free under ${\mathbb P}^{(n)}_{{\boldsymbol 0},0}={\rm P}^{(n)}_{{\boldsymbol\theta}_0, {\boldsymbol\vartheta}_0}$ for all $n$) versions of the central sequence ${\boldsymbol\Delta}^{(n)}_{\text{\rm int}}({\boldsymbol\theta}_0,f_0)$. 
\end{enumerate}

The above-mentioned properties demonstrate the considerable finite-sample advantages of ${\rm E}\big[ \Deltab\n(\thetab,f) \vert\, {\mathcal A}^{(n)}_{\scriptscriptstyle\text{{\rm Anc}}}\big]$ over classical tangent space projections $\Deltab^{(n)*}(\thetab,\hat{f}\n)$: nuisance-ancillarity also for finite sample size $n$, uniform asymptotics, and no need to estimate $f$. In more casual terms, this last property implies that the procedure based on some, possibly misspecified, assumed density $g$ always leads to valid inference, like pseudo or quasi maximum likelihood methods. When the assumed density $g$ equals the actual density $f_0$, the procedure attains the semiparametric lower bound.

\section{Conclusions}\label{sec:Conclusion}
\noindent This paper builds on a long tradition in statistics of using concepts like sufficiency and ancillarity to build optimal inference procedures. While, in particular, ancillarity is a useful concept, its application is hampered by the fact that maximal ancillary $\sigma$-fields are generally not unique. As a result, it is not clear which of the maximal ancillary $\sigma$-fields to base inference on.

We resolve this issue by introducing the notion of weak convergence of (ancillary) $\sigma$-fields and the observation that, in  limiting experiments, a unique maximal ancillary $\sigma$-field often exists. Therefore, it is natural to choose, in the sequence of (localized) experiments, a  maxi\-mal ancillary $\sigma$-field---call it  {\it strongly maximal nuisance-ancillary}---that weakly converges to that unique limiting maximal ancillary $\sigma$-field. Inference procedures that are measurable with respect to strongly maximal nuisance-ancillary  $\sigma$-fields are finite-sample nuisance-free, and  their risk functions converge to the  risk functions of the nuisance-free procedures of the limiting experiment. In the particular case of Locally Asymptotically Normal (LAN) experi\-ments, conditioning central sequences on on strongly maximal nuisance-ancillary  $\sigma$-fields yields semiparametrically efficient procedures that are finite-sample nuisance-free while traditional tangent space projections only yield {\it asymptotically} nuisance-free semiparametrically efficient procedures.

We illustrate how this approach leads to semiparametrically optimal inference in the context of multivariate  LAN experiments with unspecified innovation density, based on the measure-transportation-related concepts of center-outward ranks and signs. Nuisance-ancillarity then is distribution-freeness, and semiparametric efficiency bounds in these experiments can be achieved by  rank-based methods (e.g., tasts based on test statistics measurable with respect to center-outward ranks and signs) enjoying finite-sample distribution-freeness.

The notions on weak convergence of $\sigma$-fields and strong maximal nuisance-ancillarity, however, do not rely on the limit experiments to be Locally Asymptotically Normal, and we conjecture that the approach outlined in the present paper can be extended, {\it mutatis mutandis},  to more general situations where the limit experiments are Locally Asymptotically Mixed Normal (LAMN) or Locally Asymptotically Brownian Functional (LABF).

\bibliographystyle{imsart-nameyear}
\bibliography{references}

\renewcommand{\thesection}{\Alph{section}}
\setcounter{section}{0}

\begin{center}{\sc Appendix}\end{center}

\section{Minimal sufficiency and maximal ancillarity}\label{Basu}\hspace{-3mm}\footnote{This appendix is largely borrowed, for convenient reference, from the online supplement of \cite{hallin2021distribution}.}

This appendix collects, for ease of reference, some classical and less classical definitions and results about sufficiency and ancillarity which are scattered across Basu's papers; some of them (such as the concept of {\it strong essential equivalence}) are slightly modified to adapt our needs.   

The celebrated result commonly known as Basu's Theorem was first established as Theorem~2 in \cite{basu1955}. The same paper also contains a Theorem~1, of which Theorem~2 can be considered a partial converse. Call them Basu's First and Second Theorems, respectively. Recall that a statistics $W$ is called ancillary with respect to ${\mathcal P}$ if its distribution is identical for all ${\rm P}\in{\mathcal P}$.

\begin{prop}[Basu's First Theorem]\label{Basu1}
Let $S$ be sufficient for a family~$\cal P$ of distributions over some probability space $({\mathcal X},\mathcal A)$. Then, if a statistic $W$ is $\rm P$-independent of $S$ for all~${\rm P}\in{\mathcal P}$, it is ancillary with respect to ${\mathcal P}$.
\end{prop}

\begin{prop}[Basu's Second Theorem]\label{Basu2}
Let $T$ be (boundedly) complete and sufficient for a family~$\cal P$  of distributions over some probability space $({\mathcal X},\mathcal A)$. Then, if a statistic~$W$ is ancillary with respect to ${\mathcal P}$, it is $\rm P$-independent of $T$ for all ${\rm P}\in{\mathcal P}$. 
\end{prop}

Basu's original proof of Proposition~\ref{Basu1} was flawed, however, and Basu's First Theorem does not hold with full generality. \cite{basu1958} realized that problem and fixed it by imposing on $\mathcal P$ a sufficient additional  condition of {\it connectedness}. Some twenty years later, that condition has been replaced \citep{KandT75} with a considerably weaker necessary and sufficient one which is trivially satisfied in the context of this paper where all measures in $\mathcal P$ are assumed mutually absolutely continuous, hence share  the same null sets.
   
Recall that a sub-$\sigma$-field ${\mathcal A}_{0}$ of $\mathcal A$ such that ${\rm P}_1(A)= {\rm P}_2(A)$ for all $A\in{\mathcal A}_0$ and all~${\rm P}_1, {\rm P}_2$ in~$\mathcal P$ is called {\it ancillary}. Clearly, the $\sigma$-field ${\mathcal A}_{V}$ generated by an ancillary statistic $V$ is ancillary. Contrary to sufficient $\sigma$-fields (the smaller, the better), it is desirable for ancillary $\sigma$-fields to be a large as possible. While minimal sufficient $\sigma$-fields, when they exist, are unique,  maximal ancillary $\sigma$-fields typically exist, but are neither unique nor easily characterized---due to, among other things, null-sets issues. 
   
\cite{basu1959} therefore introduced the notions of {$\mathcal P$-essentially equivalent} and {$\mathcal P$-essentially  maximal} sub-$\sigma$-fields and established a useful sufficient condition for an ancillary statistic to be essentially  maximal. In this paper, we avoid these null-sets complications by restricting to families $\mathcal P$ of mutually absolutely continuous  distributions which, therefore, are sharing the same null sets. This yields the following version of Basu's result---call it Basu's Third Theorem---which is be sufficient for our needs.

      
\begin{prop}[Basu's Third Theorem for families of mutually absolutely conti\-nuous distributions]\label{Basu3}
Let $\mathcal P$ be a family of mutually absolutely continuous distributions over~$({\mathcal X},{\mathcal A})$. Denote by~${\mathcal A}_{\text{\rm suff}}$ a  (boundedly) complete and sufficient (for $\mathcal P$) sub-$\sigma$-field of~$\mathcal A$. Then, any ancillary sub-$\sigma$-field ${\mathcal A}_{\text{\rm anc}}$ (i) containing all $\mathcal P$-null sets, and (ii) such that $\sigma\big({\mathcal A}_{\text{\rm suff}}\cup {\mathcal A}_{\text{\rm anc}}\big)=\mathcal A$  is 
maximal ancillary. 
\end{prop}
      
      


\newcommand{\expect}{{\mathbb E}}
\section{Completeness of nuisance experiments}$\,$

Consider an experiment $\E \coloneqq \big(\outcome, \Field, \P\big)$ with $\P \coloneqq \left\{\prob_{\lpnui}: \lpnui\in\LPnui \right\} .$ Throughout  assume that the probability measures in $\P$ are mutually absolutely continuous and denote by $\expect_\lpnui$ the expectation under $\prob_\lpnui$. 

Recall that   the $\sigma$-field ${\cal B}_0\subseteq{\cal B}$ is called (boundedly) \emph{complete} in $\E$ if, for any (bounded) real-valued ${\Field}_0$-measurable 
 $Y$,   $\expect_\lpnui Y=0$ for all $\lpnui\in\LPnui$  implies~$Y=0$ $\prob_0$-a.s. (hence, $\prob_\eta$-a.s.\ for all $\eta\in\Upsilon$). Equivalently, we also say that the {\em marginal} subexpe\-riment~$\E_{{\cal B}_0}\coloneqq \big(\outcome, {\cal B}_0, \P \big)$ 
  (in which the statistician does not observe $\Field$, but only ${\cal B}_0$)   is (boundedly) \emph{complete}.  
 Denote~by 
\begin{equation}\label{eqn:spanRNDervatives}
	 {\cal S}_{{\cal B}_0}\coloneqq \text{\rm span}\left\{
	\expect_0\left[\dfrac{\dd\prob_{\lpnui}}{\dd\prob_{0}}\vert
	{\cal B}_0\right]: \lpnui\in\LPnui\right\}
\end{equation}
the subspace of $L^1({\cal B}_0,  {\mathbb P}_{0})$ spanned by the marginal likelihood ratios 
 $\expect_0\left[{\dd\prob_{\lpnui}/\dd\prob_{0}}\vert
{\cal B}_0\right]$  and note that  $ {\cal S}_{{\cal B}_0}\subseteq  {\cal S}_{{\cal B}}$. It follows from \citet[Proposition~1]{hallin2023bounded}---an extension of \cite{Far62} to restricted subexperiments---that ${\cal B}_0$ (equivalently, the subexperiment~${\E}_{{\cal B}_0}$) is complete if and only if~${\cal S}_{{\cal B}_0}$ is dense in $L_1({\Field}_{0},\prob_0)$. 
%

Recall that a sequence $\left(\FieldSub_k\right)_{k\in\rN}$ of $\sigma$-fields  is a \emph{filtration} on $(\Omega,{\Field}_0)$ if $\FieldSub_k\subseteq~\!\FieldSub_{k+1}\subseteq~\!{\Field}_0$ for  all~$k\in\rN$; call it  {\it convergent } if $\bigcup_{k\in\rN}\FieldSub_k={\Field}_0$. The following result is convenient in proving bounded completeness of experiments defined on infinite-dimensional sample spaces. 

\begin{lemma}\label{lem:CompletenessInfinite}
	Consider the experiment $\E = \big(\outcome, \Field, \P = \left\{\prob_{\lpnui}: \lpnui\in\LPnui \right\} \big)$ and let $\left(\FieldSub_k\right)_{k\in\rN}$ be a convergent filtration on $(\Omega,{\Field}_0)$ where ${\Field}_0\subseteq{\cal B}$. 
	Then ${\cal E}_{{\cal B}_0}$ is boundedly complete 
	if and only if each marginal subexperiment $\E_{\FieldSub_k}\coloneqq \big(\outcome, {\cal F}_k, \P \big)$ is boundedly complete.
\end{lemma}
\begin{proof}
	The ``only if\," part is trivial. As for the ``if\," part, in view of \citet[Proposition~1]{hallin2023bounded}, it is sufficient to establish that if each~$\E_{\FieldSub_k}$ is boundedly complete, then  ${\cal S}_{{\cal B}_0}$ is dense in $L_1({\Field}_{0},\prob_0)$, i.e., for any  bounded  $Y\in L_1(\Field,\prob_0)$ and   $\delta>~\!0$, there exists $Y^\prime \in {\cal S}_{{\cal B}_0}$ such that 
	${\mathbb E}_0[\vert Y^\prime -Y\vert]<~\!\delta$. 
 From Doob's martingale convergence theorem, we know 
 that~$Y_k\coloneqq\expect_0\left[Y\vert\FieldSub_k\right]$ converges to~$Y$ in~$L_1(\Field,\prob_0)$. Moreover, since $Y$ is bounded, $Y_k\in{\cal S}_{{\cal F}_k}$. As a result, we can find $k\in\rN$ and $Y''$ in $L_1(\FieldSub_k,\prob_0)$ such that $Y''$ is closer than $\delta/2$ to $Y$. Since $\E_{\FieldSub_k}$ is complete, $S_{{\cal F}_k}$ is dense in $L_1(\FieldSub_k,\prob_0)$, we  subsequently can find $Y'\in S_k$ closer than $\delta/2$ to $Y''$. As $S_k\subset S$, we found $Y'\in S$ closer than $\delta$ to $Y$. Mazur's theorem concludes the proof.
\end{proof}
\end{document}